\documentclass[11pt]{article}
\usepackage{mathrsfs}
\usepackage{amsfonts}
\usepackage{amssymb,amsmath,amsthm}
\newcommand{\R}{\mathbb{R}}

\newcommand{\N}{\mathbb{N}}
\newcommand{\beq}{\begin{equation}}
\newcommand{\eeq}{\end{equation}}
\newcommand{\bp}{\begin{proof}}
\newcommand{\ep}{\end{proof}}
\newcommand{\bo}{\begin{proposition}}
\newcommand{\eo}{\end{proposition}}
\newcommand{\bl}{\begin{lemma}}
\newcommand{\el}{\end{lemma}}

\newtheorem{theorem}{Theorem}[section]
\newtheorem{corollary}{Corollary}

\newtheorem{lemma}[theorem]{Lemma}
\newtheorem{proposition}{Proposition}

\theoremstyle{definition}

\newtheorem{remark}{Remark}

\begin{document}
\title{\LARGE\bf{Uniqueness and nondegeneracy of ground states for the Schr\"{o}dinger-Newton equation with
power nonlinearity}$\thanks{{\small This work was partially supported by NSFC(11901532).}}$ }
\date{}
 \author{ Huxiao Luo$\thanks{{\small Corresponding author. E-mail: luohuxiao@zjnu.edu.cn (H. Luo).}}$\\
\small Department of Mathematics, Zhejiang Normal University, Jinhua, Zhejiang, 321004, P. R. China
}
\maketitle
\begin{center}
\begin{minipage}{13cm}
\par
\small  {\bf Abstract:} In this article, we study the Schr\"{o}dinger-Newton equation
\begin{equation}\label{20230913-1}
-\Delta u+\lambda u=\frac{1}{4\pi}\left(\frac{1}{|x|}\star u^{2}\right)u+|u|^{q-2}u \quad \text{in}~\R^3,
\end{equation}
where $\lambda\in\R_+$, $q\in (2,3)\cup(3, 6)$. By investigating limit profiles of ground states as $\lambda\to0^+$ or $\lambda\to+\infty$, we prove the uniqueness of ground states. By the action of the linearized eqaution with respect to decomposition into spherical harmonics, we obtain the nondegeneracy of ground states.
 \vskip2mm
 \par
 {\bf Keywords:} Schr\"{o}dinger-Newton equation; Uniqueness and nondegeneracy of ground states; Spherical harmonics.
 \vskip2mm
 \par
 {\bf MSC(2010): } 35A02, 35B20, 35J15, 35J20, 35J61

\end{minipage}
\end{center}

 {\section{Introduction}}
 \setcounter{equation}{0}
Consider the time dependent Schr\"{o}dinger equation with combined nonlinearities
\begin{equation}\label{t1.1.0}
\left\{
\begin{array}{ll}
\aligned
&i\frac{\partial\psi}{\partial t}+\Delta \psi+v(x)\psi+\mu|\psi|^{q-2}\psi=0,\quad t\in\R,~ x\in\R^3,\\
&-\Delta v=|\psi|^2, \\
&\psi(0,x)=\psi_0(x),
\endaligned
\end{array}
\right.
\end{equation}
where $\mu\in\R$ is a parameter, $\psi$ is the wave function, the local term $|\psi|^{q-2}\psi$ arises from the effects of
the short-range self-interaction between particles, the nonlocal term $v: \R^3 \to \R$ represents
the Newtonian gravitational attraction between particles. (\ref{t1.1.0}), also known as Gross-Pitaevskii-Poisson equation, is used to describe the dynamics of the Cold Dark Matter in the form of the Bose-Einstein Condensate \cite{4,8,39,45}.

For the non-interacting case $\mu = 0$, (\ref{t1.1.0}) is  usually called the Hartree or Schr\"{o}dinger-Newton equation, which is used to describe the quantum mechanics of a polaron at rest \cite{Pekar}. In 1976, Choquard used the Hartree equation to describe an electron trapped in its own hole in a certain approximating to Hartree-Fock theory of one component plasma, see e.g. E. H. Lieb \cite{Lieb1976/77}. Thus, the Hartree equation is also called Choquard equation. The general Choquard equation is mathematically well-studied by V. Moroz and J. Van Schaftingen in \cite{MR3056699,MR3356947,MR3625092}.

Let $\psi(t, x) = e^{i\lambda t}u(x)$ be the standing wave for problem (\ref{t1.1.0}) with the focusing case $\mu = 1$, where $\lambda\in\R$ is the frequency. Then $u(x)$ solves the following elliptic system
\begin{equation}\label{eq:20220603-1}
\left\{
\begin{array}{ll}
\aligned
&-\Delta u+\lambda u=vu+|u|^{q-2}u \quad \text{in}~\R^3, \\
&-\Delta v=u^2.
\endaligned \tag{Q}
\end{array}
\right.
\end{equation}
As usual, we consider the solutions of system \eqref{eq:20220603-1} in Sobolev space $H^1(\R^3)\times \dot{H}^1(\R^3)$. Since $-\Delta v = u^2$ has a unique positive solution $v=I_2\star u^2$, \eqref{eq:20220603-1} is equivalent to the following elliptic equation
\begin{equation}\label{eq:20220420-e1}
-\Delta u+\lambda u=\left(I_2\star u^{2}\right)u+|u|^{q-2}u \quad \text{in}~\R^3, \tag{P}
\end{equation}
where $I_2(\cdot):=\frac{1}{4\pi}\frac{1}{|\cdot|}$ is the Green function of the Laplacian $-\Delta$ on $\R^3$.

In studying the dynamics around the ground state, the uniqueness and nondegeneracy of ground state play a crucial role. The uniqueness and nondegeneracy of ground state also have applications to the study of nonlinear elliptic equations. This is our main motivation for the present paper.

In the case of a single power nonlinearity, the uniqueness and nondegeneracy of positive solutions to the Schr\"{o}dinger equation
 \beq\label{20220609-V1}
-\Delta u+u= u^{q-1}~~\hbox{in}~\R^3,~q\in(2,6)
\eeq
have been proved by Kwong in his celebrated paper \cite{MR969899}.
The uniqueness of the positive solution for the Schr\"{o}dinger-Newton equation
 \beq\label{20220609-VU3}
-\Delta u+u=\left(I_2\star u^2\right)u~~\hbox{in}~\R^3
\eeq
 also has been solved, see in \cite{Lieb1976/77,Ma2010,MR1677740,MR2492602}. The nondegeneracy of the solution for (\ref{20220609-VU3}) was proved by Lenzmann \cite[Theorem 1.4]{Lenzmann2009}, see also Tod-Moroz \cite{MR1677740} and Wei-Winter \cite{MR2492602}.

When $q=3$, G. Vaira \cite{Vaira2013} proved that all positive solutions of
  \beq\label{20220609-V2}
-\Delta u+u=\left(I_2\star u^2\right)u+ u^2~~\hbox{in}~\R^3
\eeq
are radially symmetric and the linearized operator around a radial ground state is non-degenerate. However, the uniqueness of ground state solutions for \eqref{20220609-V2}
is still an open problem. Indeed, $q=3$ is Coulomb-Sobolev critical exponent, see e.g. \cite{lions1987solutions,MR3568051,MR2679375} for details.

In \cite{1,2}, T. Akahori et al. studied the ground state for the following Schr\"{o}dinger equation
\beq\nonumber
-\Delta u + \lambda u = |u|^{\frac{2N}{N-2}-2}u + |u|^{q-2}u~~\text{in}~\R^N,
\eeq
 where $N \geq 3$, $2 < q < \frac{2N}{N-2}$. T. Akahori et al.\cite{2} proved that if $N \geq 4$, the ground state is unique when $\lambda$ is sufficiently small. And as $\lambda\to 0$, the unique ground state $u$ tends to the unique positive solution of the the Schr\"{o}dinger equation
 \eqref{20220609-V1}.
In \cite{louis}, J. Louis, J. Zhang and X. Zhong studied the following Schr\"{o}dinger equation
\beq\label{eq:20231106-1}
-\Delta u + \lambda u = g(u)~~\text{in}~\R^N
\eeq
under the following assumptions:
\begin{itemize}
\item[(G1)] $g \in C^1(\R)$, $g(s)>0$ for $s> 0$; $g(s)=0$ for $s\leq0$;
\item[(G2)] there exists some $(\alpha,\beta)\in (2,6)$ satisfying
such that
$$\lim_{s\rightarrow 0^+}\frac{g'(s)}{s^{\alpha-2}}=\mu_1(\alpha-1)>0 \quad \mbox{and} \quad \lim_{s\rightarrow +\infty}\frac{g'(s)}{s^{\beta-2}}=\mu_2(\beta-1)>0.$$
\end{itemize}
By studying the asymptotic behavior of positive solutions as $\lambda\to 0$ and $\lambda\to\infty$ both in $C_{r,0}(\R^N)$ and $H^1(\R^N)$, the authors obtained the uniqueness of positive solutions to \eqref{eq:20231106-1} when $\lambda$ sufficiently close $0$ or $+\infty$. Here, $C_{r,0}(\R^N)$ denotes the space of continuous radial functions vanishing at $\infty$.

However, as far as we know, there is no uniqueness and nondegeneracy result about the ground
states for \eqref{eq:20220603-1}(or the equivalent equation \eqref{eq:20220420-e1}). This might be due to the fact that the study of uniqueness and nondegeneracy of ground states for nonlocal problem \eqref{eq:20220603-1}(or the equivalent equation \eqref{eq:20220420-e1}) requires some work far from trivial. To fill up this gap, in this article, we study the uniqueness and nondegeneracy of ground states to \eqref{eq:20220603-1} when $\lambda$ sufficiently close $0$ or $+\infty$.

Inspired by \cite{louis}, we study the asymptotic behaviour of positive solutions as $\lambda\to 0$ and $\lambda\to\infty$, and then prove the uniqueness result.
\begin{theorem}\label{th1.2}
 Let $q\in(2,3)\cup(3,6)$. Then, there exist $\lambda^*, \lambda_*>0$ such that
for any $\lambda>\lambda^*$ or $0<\lambda<\lambda_*$,  the positive ground state to \eqref{eq:20220603-1} is unique in $H^1(\R^3)$.
\end{theorem}
\begin{remark}
Compared with the Schr\"{o}dinger equation studied in \cite{louis}, there are some new difficulties when we study the asymptotic behavior of ground state solutions. The non-local characteristic of \eqref{eq:20220420-e1} prevents us from obtaining a priori estimation of solutions, if we adopt the blow up procedure used in \cite{louis}. To overcome this difficult, we use the system  \eqref{eq:20220603-1}
 equivalent to equation \eqref{eq:20220420-e1}, and obtain uniform boundedness by using the double blow-up method.

 Another problem is that the non-local characteristic of \eqref{eq:20220420-e1} prevents us from obtaining uniform decay of solutions, we use the Newton's theorem for radially symmetric functions  \cite[(9.7.5)]{LiebLoss.2001} to overcome this difficulty, see in Lemma \ref{20230413-l1} and Lemma \ref{lemma:20220423-l1}.
\end{remark}

The second goal of this article is to prove the nondegeneracy of ground states to \eqref{eq:20220603-1}. Let $(U_\lambda, V_\lambda)$ be a ground state for \eqref{eq:20220603-1}. We see that
its derivatives $(\partial_i U_\lambda, \partial_i V_\lambda)$ are solution of the linearized system
\beq\label{eq:20230419-e1}
 \left\{
\begin{array}{ll}
\aligned
&-\Delta\xi + \lambda\xi = V_\lambda\xi+\zeta U_\lambda+(q-1)U_\lambda^{q-2}\xi,\\
&-\Delta\zeta= 2\xi U_\lambda.
\endaligned
\end{array}
\right.
\eeq
 Define the linear operator $\mathcal{L}_{\lambda}^+$ associated to $(U_{\lambda}, V_{\lambda})$ by
  \beq\label{eq:20220913-e2}
\mathcal{L}_{\lambda}^+ \left(\begin{matrix} \xi\\ \zeta \end{matrix} \right)  =\left(\begin{matrix} -\Delta\xi + \lambda\xi- V_{\lambda} \xi - \zeta U_{\lambda}-(q-1)U_\lambda^{q-2}\xi \\ -\Delta \zeta- 2\xi U_{\lambda}\end{matrix} \right).
\eeq
Obviously, $(\partial_i U_\lambda, \partial_i V_\lambda)\in \text{Ker}(\mathcal{L}_{\lambda}^+)$. If these derivatives and their linear combinations exhaust the kernel of the operator $\text{Ker}(\mathcal{L}_{\lambda}^+)$, then the ground state solutions are called nondegenerate. We use the spherical harmonics method to prove the nondegeneracy.
\begin{theorem}\label{th1.3}
 Let $q\in(2,3)\cup(3,6)$. Suppose that $(\xi, \zeta) \in H^2(\R^3) \times H^2(\R^3)$ satisfies the eigenvalue problem \eqref{eq:20230419-e1}. Then, there exist $\lambda^*, \lambda_*>0$ such that
for any $\lambda>\lambda^*$ or $0<\lambda<\lambda_*$,
$$(\xi, \zeta) \in  \text{span}\{(\partial_i U_\lambda, \partial_i V_\lambda), i = 1,2, 3\}.$$
\end{theorem}
\begin{remark}
The spherical harmonics method is used in studying the nondegeneracy of ground states for the Schrodinger equation, see e.g. \cite{FV}. Compared with \cite{FV}, due to the presence of the Hartree term, we can't use the Perron-Frobenius-type arguments in \cite{FV} directly.
To overcome this difficulty, we rely on more complex analysis for the action of the linearized system \eqref{20221219-e1} with respect to decomposition into spherical harmonics, see Lemma \ref{3.9}.
\end{remark}

The paper is organized as follows. In Section 2, we give the existence, regularity and radial symmetry of ground state solutions for equation \eqref{eq:20220603-1}, and some related Liouville type theorems which will be used in blow up analysis of ground state solutions. In Section 3 we address a priori bounds of ground state solutions and asymptotic behaviors of solutions in $C_{r,0}(\R^3)$, as $\lambda\to0^+$ or $\lambda\to+\infty$. In Sections 4-5, we prove the uniqueness result ( Theorem \ref{th1.2}) and the nondegeneracy result (Theorem \ref{th1.3}), respectively.

\vskip4mm
{\section{Preliminaries and Background Results}}
 \setcounter{equation}{0}

\vskip4mm
\subsection{Existence, regularity and radial symmetry of solutions}

The associate energy functional for \eqref{eq:20220420-e1} is defined as
\begin{equation}\label{J:20220606}
J_\lambda(u):=\frac{1}{2}\|\nabla u\|_2^2+\frac{\lambda}{2}\|u\|_2^2-\frac{1}{4}\int_{\R^3}\left(I_2\star u^2\right)u^2dx-\frac{1}{q}\int_{\R^3}|u|^qdx,\quad \forall u\in H^1(\R^3).
\end{equation}
The energy functional $J_\lambda$ is well defined in $H^1(\R^3)$, thanks to the following Hardy--Littlewood--Sobolev inequality ( or abbreviated H-L-S inequality).
\begin{proposition}\cite{MR3625092} Let $t$, $r>1$ and $0<\alpha<N$ with $\frac{1}{t}+\frac{1}{r}=1+\frac{\alpha}{N}$, $f\in L^{t}(\R^N)$ and $h\in L^{r}(\R^N)$. There exists a constant $C(N,\alpha,t,r)$, independent of $f,h$, such that
\begin{equation*}
\left\|I_\alpha\star h\right\|_{t'}\leq C(N,\alpha,t,r)\|h\|_{L^r(\R^N)}
\end{equation*}
and
\begin{equation*}
\int_{\R^N}\left(I_\alpha\star h\right)f dx\leq C(N,\alpha,t,r) \|f\|_{L^t(\R^N)}\|h\|_{L^r(\R^N)},
\end{equation*}
where $t'$ denotes the conjugate exponent such that $\frac{1}{t'}+\frac{1}{t}=1$.
\end{proposition}
Indeed, by H-L-S inequality, H\"{o}lder's inequality and Sobolev inequality, for any $u,v\in H^1(\R^3)$ we have
\begin{equation}\label{20230418-ea1}
\int_{\R^3}\left(I_2\star u^2\right)u^2 dx\leq C \|u\|^4_{\frac{12}{5}}<\infty.
\end{equation}

A nontrivial solution
$u \in H^1(\R^3)$ of \eqref{eq:20220420-e1} is called a ground state solution (or least action solution) if
$$J_\lambda(u) = c_\lambda^*:= \inf\limits_{w\in \mathcal{M}_\lambda}J_\lambda(w),$$
where
\beq\label{20230417-e2}
\mathcal{M}_\lambda:=\{w\in H^1(\R^3)\setminus\{0\}: ~~J'_\lambda(w) = 0\}.
\eeq

\begin{proposition}\label{pro:20220606} (\cite[Theorem 1.1]{LMZ}) The problem \eqref{eq:20220420-e1} admits a nontrivial solution $u \in H^1(\R^3)$ at the level $c_\lambda=c_\lambda^*$,
where \begin{equation}\label{MP1}
c_\lambda:=\inf\limits_{\gamma\in \Gamma_\lambda}\max\limits_{t\in[0,1]}J_\lambda(\gamma(t)),
\end{equation}
where
\begin{equation}\label{MP2}
\Gamma_\lambda:=\left\{\gamma\in C([0,1], H^1(\R^3)): \gamma(0)=0, J_\lambda(\gamma(1))<0\right\}.
\end{equation}
 Moreover, $u \in W^{2,s}(\R^3)$ for every $s > 1$ and $u$ satisfies the Pohozaev identity
\beq\label{eq:20220606-p}
\mathcal{P}_\lambda(u):=\frac{1}{2}\|\nabla u\|_2^2+\frac{3}{2}\lambda \|u\|_2^2-\frac{5}{4}\int_{\R^3}\left(I_2\star u^2\right)u^2-\frac{3}{q}\int_{\R^3}|u|^qdx=0.
\eeq
\end{proposition}
\begin{remark} Indeed, Theorem 1.1 in \cite{LMZ} is a general result than Proposition \ref{pro:20220606}.
In \cite[Appendix]{LMZ}, the authors give some conditions such that Proposition \ref{pro:20220606} hold.
In fact, the specific parameters in our setting, $N = 3$, $\alpha=2$ and $p\in(2,3)\cup(3,6)$, satisfy
$(A1')$ and $(A8')$ in \cite[Appendix]{LMZ}.
\end{remark}
\begin{remark}
The positive of solution $u$ can be obtained similarly to that of \cite{LMZ,MR3356947}. In what follows, we use $u^{q-1}$ instead of $|u|^{q-2}u$.
\end{remark}

\begin{corollary} \label{2023411-cro1} Let $u\in H^1(\R^3)$ be a solution to \eqref{eq:20220420-e1}, then
 $u\in C^{1,\sigma}(\R^3)$ for any $\sigma\in(0,1)$.
\end{corollary}
\bp
From Proposition \ref{pro:20220606}, $u \in W^{2,s}(\R^3)~\forall s\in[1, +\infty)$, then by Sobolev embedding, the conclusion holds.
\ep

\begin{corollary}\label{MP-20220607} Let $u$ be the solution obtained in Proposition \ref{pro:20220606}, then
 $c_\lambda$ satisfies the identity:
\beq\label{eq:20220606-c}
c_\lambda=J_\lambda(u) = \frac{1}{3}\int_{\R^3}|\nabla u|^2dx+\frac{1}{6}\int_{\R^3}\left(I_2\star u^2\right)u^2dx.
\eeq
And $\lambda \mapsto c_\lambda$ is non-decreasing.
\end{corollary}
\begin{proof}
We deduce \eqref{eq:20220606-c} from \eqref{eq:20220606-p} and (\ref{J:20220606}).
$\lambda \mapsto c_\lambda$ is non-decreasing due to the mountain
pass characterization \eqref{MP1}-\eqref{MP2}.
\end{proof}

The radially symmetry and monotone decreasing property of $u$ can be obtained similarly to that of \cite{MR3356947} by the theory of polarization, or \cite[Theorem 3.1]{Vaira2013} by moving plane method.
\begin{proposition}\label{p2.2}
 Let $(U_\lambda, V_\lambda)\in H^1(\R^3)\times \dot{H}^1(\R^3)$ be a positive solution to \eqref{eq:20220603-1}. Then, up to a translation, $U_\lambda$ and $V_\lambda$ are radially symmetric and decreases with respect to $|x|$.
\end{proposition}

\vskip4mm
\subsection{Some Liouville type theorems}

In order to get a priori bounds of $u_{\lambda}$ for \eqref{eq:20220420-e1}, we use a blow up procedure introduced by Gidas and Spruck \cite{MR619749,MR615628}. To this, we need some Liouville theorems. The first one is the typical nonlinear Liouville theorem, which goes back to J. Serrin in the 1970s.
\bl\label{lemma:20220603-l1} (\cite{MR615628}) The equation
$$-\Delta u = u^q,~~ u\geq 0,~~ x \in \R^N ,~N\geq2$$
has no nontrivial global classical solution if $q < \frac{N+2}{N-2}$.
\el

\bl\label{lemma:20220610-l1} (\cite[Theorem 8.4]{QuittnerSouplet.2007}) The problem
$$-\Delta u \geq u^q,~~ u\geq 0,~~ x \in \R^N ,~N\geq2$$
has no nontrivial global classical solution if $1<q \leq \frac{N}{N-2}$.
\el

\bl\label{lemma:20220420-l1}
There is no nontrivial nonnegative classical solution to
\beq\nonumber
-\Delta u\geq \left(\frac{1}{|x|}\star u^2\right)u~~\hbox{in}~\R^3.
\eeq
\el
\bp
We first remark that if $u\not\equiv 0$ is a  nonnegative solution, then $u(x)$ is positive in $\R^3$ by the maximum principle.
By comparison with the harmonic function $|x|^{-1}$, for any $\rho>0$, we obtain that there exists some $C_\rho>0$ such that
\beq\nonumber
u(x)\geq C_\rho |x|^{-1}, ~~\forall |x|\geq \rho.
\eeq
Then for any $x\neq 0$, we have that
\begin{align*}
\int_{\R^3}\frac{1}{|x-y|}u^2(y)dy\geq & \int_{|y|\geq 2|x|} \frac{1}{|x-y|}u^2(y)dy\\
\geq & \int_{|y|\geq 2|x|} \frac{2}{3}\frac{1}{|y|}u^2(y)dy
\geq C_{|x|} \int_{2|x|}^{+\infty} \frac{1}{r}dr=+\infty.
\end{align*}
And for $x=0$, we also have that $\int_{\R^3} \frac{1}{|y|}u^2(y)dy=+\infty$.
Hence, we obtain that
$$-\Delta u\geq \left(\frac{1}{|x|}\star u^2\right)u\geq u~\hbox{in}~\R^3,$$
which also implies that $u\equiv 0$ by J. Serrin's Liouville theorem \cite[Theorem 8.4]{QuittnerSouplet.2007}.
\ep

Indeed, the conclusion of Lemma \ref{lemma:20220420-l1} above is covered by \cite[Theroem 4.1]{MR3625092}.
In addition to the above lemmas, we also need the following Liouville type theorems about elliptic systems.

\bl\label{lemma:20220611-l1} \cite[Theorem 1.4-(i)]{quittner2012symmetry}
 Let $p, q, r, s \geq 0$ and $N\geq 3$.
Assume $p - s = q - r \geq 0.$ Then any nonnegative classical solution $(u, v)$ of
\begin{equation}\label{eq:20220611-1}
\left\{
\begin{array}{ll}
-\Delta u= u^rv^p~~\hbox{in}~\R^N, \\
-\Delta v=v^su^q~~\hbox{in}~\R^N.
 \end{array}
\right.
\end{equation}
 satisfies $u \geq v$ or $v \geq u$.
\el

\bl\label{lemma:20220601-l1} Let $u\geq0$ and $v\geq0$
be classical solution to
\begin{equation*}
\left\{
\begin{array}{ll}
-\Delta u= vu~~\hbox{in}~\R^3, \\
-\Delta v=u^2~~\hbox{in}~\R^3.
 \end{array}
\right.
\end{equation*}
Then $u\equiv0$. If we further assume $\lim\limits_{|x|\to+\infty}v(x)=0$, then $v\equiv0$.
\el
\begin{proof}
 By using Lemma \ref{lemma:20220611-l1}, let $r=1$, $p=1$, $s=0$ and $q=2$, we get $u\geq v$ for all $x\in\R^3$ or $v\geq u$ for all $x\in\R^3$.

 Case $u\geq v$. In this case, we have $-\Delta v\geq v^2~\hbox{in}~\R^3$,
 then by Lemma \ref{lemma:20220610-l1} we get $v\equiv0$, and obviously $u\equiv0$.

 Case $v\geq u$. In this case, we have $-\Delta u\geq u^2~\hbox{in}~\R^3$,
 then by Lemma \ref{lemma:20220610-l1} we get $u\equiv0$. Thus $-\Delta v=0$. Thus, by $\lim\limits_{|x|\to+\infty}v(x)=0$, we get $v\equiv0$.
\end{proof}

\vskip4mm
{\section{ Uniform estimates and asymptotics } }
 \setcounter{equation}{0}

\vskip4mm
\subsection{A priori bounds and compactness results}

Let $u_\lambda\in H^1(\R^3)$ be the ground state solution for \eqref{eq:20220420-e1}. Then $(u_\lambda, v_\lambda)\in H^1(\R^3)\times \dot{H}^1(\R^3)$ is the ground state solution for \eqref{eq:20220603-1},
where $v_\lambda=I_2\star u_\lambda^2$.
\bl\label{2023412-l1} For $\lambda\in(0, \Lambda)$, then $\int_{\R^3} |\nabla u_\lambda|^2+|u_\lambda|^2 dx\leq C(\Lambda)$.
\el
\bp First, by Corollary \ref{MP-20220607}, we have
$
c_{\lambda}\leq c_{\Lambda}.
$
If $2<q\leq4$, we have from
\begin{equation}\nonumber
\aligned
qc_\lambda&=qJ_\lambda(u_\lambda)-\langle J'_\lambda(u_\lambda),u_\lambda\rangle \\
&=\left(\frac{q}{2}-1\right)\int_{\R^3}|\nabla u_\lambda|^2+\lambda|u_\lambda|^2 dx+\left(1-\frac{q}{4}\right)\int_{\R^3}\left(I_2\star u_\lambda^2\right)u_\lambda^2 dx
\endaligned
\end{equation}
that $\{u_\lambda\}$ is bounded in $H^1(\R^3)$.
If $4<q<6$, we have from
\begin{equation}\nonumber
\aligned
4c_\lambda&=4J_\lambda(u_\lambda)-\langle J'_\lambda(u_n),u_n\rangle \\
&=\int_{\R^3}|\nabla u_\lambda|^2+\lambda|u_\lambda|^2 dx+\left(1-\frac{4}{q}\right)\int_{\R^3}u_\lambda^q dx
\endaligned
\end{equation}
that $\{u_\lambda\}$ is also bounded in $H^1(\R^3)$.
Thus, the conclusion holds.
\ep

\bl\label{20230122-l1} For $\lambda\in(0, \Lambda)$, then $\int_{\R^3} |\nabla v_\lambda|^2 dx\leq C(\Lambda)$.
\el
\bp
By Corollary \ref{MP-20220607}, we have
\begin{equation}\label{e:20230122-a2}
\|\nabla v_{\lambda}\|_2^2=\int_{\R^3}(-\Delta v_{\lambda})v_{\lambda} dx=\int_{\R^3} u_\lambda^2 v_\lambda dx=\int_{\R^3}\left(I_2\star u_\lambda^2\right) u_\lambda^2  dx\leq 6c_{\lambda}\leq 6c_{\Lambda}.
\end{equation}
Thus, the conclusion holds.
\ep

\bl\label{lemma:prior-estimate1}
For any $M>0$, there exists $ C(M)>0$ such that, for any non-negative solution $(u,v)\in H^1(\R^3)\times\dot{H}^1(\R^3) $ to \eqref{eq:20220603-1} with $\lambda\in (0,M)$,
$$\max_{x\in \R^3}u(x)\leq C(M)\quad\text{and}\quad\max_{x\in \R^3}v(x)\leq C(M).$$
\el
\begin{proof}
We  proceed  by contradiction, assuming there exists a sequence
$(\lambda_n, u_n, v_n) \in (0, M) \times H^1(\mathbb{R}^3)\times \dot{H}^1(\R^3)$ where $(u_n,v_n)\in H^1(\R^3)\times\dot{H}^1(\R^3) $ to \eqref{eq:20220603-1}
with $\lambda=\lambda_n$ and
\begin{align*}
\max_{x\in \R^3}u_n(x)\rightarrow +\infty\quad\text{or}\quad\max_{x\in \R^3}v_n(x)\rightarrow +\infty,\quad \mbox{as}~ n\rightarrow +\infty.
\end{align*}
By Proposition \ref{p2.2}, we see that $u_n$ and $v_n$ are positive radial decreasing functions, and
$$u_n(0)=\max_{x\in \R^3}u_n(x)\quad \text{and}~~v_n(0)=\max_{x\in \R^3}v_n(x).$$
We follow a blow up procedure introduced by Gidas and Spruck \cite{MR615628}. Let
$$M_n:=u_n(0)+v_n(0)\rightarrow +\infty,\quad \mbox{as} \, n\rightarrow +\infty.$$
Consider $\tilde{u}_n(y)=\frac{1}{M_n}u_n(M_n^\sigma y)$ and $\tilde{v}_n(y)=\frac{1}{M_n}v_n(M_n^\sigma y)$, by a direct computation,
\begin{equation}\label{eq:20220420-e3}
\left\{
\begin{array}{ll}
\aligned
&-\Delta \tilde{u}_n=M_{n}^{1+2\sigma}\tilde{v}_n\tilde{u}_n+M_n^{q-2+2\sigma}\tilde{u}_n^{q-1}-\lambda_n M_{n}^{2\sigma}\tilde{u}_n \quad \text{in}~\R^3, \\
&-\Delta \tilde{v}_n=M_{n}^{1+2\sigma}\tilde{u}_n^2.
\endaligned
\end{array}
\right.
\end{equation}
Note that $\|\tilde{u}_n\|_\infty\leq1$, $\|\tilde{v}_n\|_\infty\leq1$ and $\lambda_n\in(0,M)$, we consider the limit of \eqref{eq:20220420-e3} for two cases.

{\bf Case:} $3<q<6$, we take $\sigma=-\frac{q-2}{2}<0$, then $1+2\sigma=3-q<0$ and $q-2+2\sigma=0$.

One can see that the right hand side of \eqref{eq:20220420-e3} is uniformly $L^\infty(\R^3)$-bounded. Applying a standard elliptic estimate, and passing to a subsequence if necessary, we may assume that
$$\tilde{u}_n\rightarrow \tilde{u},\quad \tilde{v}_n\rightarrow \tilde{v} \quad \text{in}~~ C_{loc}^{2}(\R^3),$$
where $\tilde{u}$ is a non-negative bounded radial solution to
\begin{equation*}
\left\{
\begin{array}{ll}
\aligned
&-\Delta \tilde{u}=\tilde{u}^{q-1}~\hbox{in}~\R^3,\\
&-\Delta \tilde{v}=0.
\endaligned
\end{array}
\right.
\end{equation*}
Lemma \ref{lemma:20220603-l1} implies $\tilde{u}\equiv0$.

Next, we show $\tilde{v}\equiv0$. Since $\tilde{v}_n\rightarrow \tilde{v}$ in $C_{loc}^{2}(\R^3)$ cannot guarantee $\lim\limits_{|x|\to+\infty}\tilde{v}(x)=0$, we cannot obtain from $-\Delta \tilde{v}=0$ that $\tilde{v}\equiv0$ directly. However, from Lemma \ref{20230122-l1}, $v_n$ is bounded in $\dot{H}^1(\R^3)$. Then by
$$\|\tilde{v}_n\|_{\dot{H}^1}=M_n^{-\frac{\sigma}{2}-1}\|v_n\|_{\dot{H}^1},$$
and $-\frac{\sigma}{2}-1=\frac{q-6}{4}<0$, $\tilde{v}_n$ is bounded in $\dot{H}^1(\R^3)$. Up to subsequence, we can assume that
$$\tilde{v}_n \rightharpoonup \bar{v}~~ \text{in}~~\dot{H}^1(\R^3);\quad \tilde{v}_n(x)\to \bar{v}(x)~~ \text{a.e.}~~\R^3.$$
By the uniqueness of limits, we have $\tilde{v}=\bar{v}\in \dot{H}^1(\R^3)$.
Then by $-\Delta \tilde{v}=0$, we get $\tilde{v}\equiv0$.

Therefore, $\tilde{u}=\tilde{v}\equiv0$, which contradict with $\|\tilde{u}\|_\infty+\|\tilde{v}\|_\infty=1$.

{\bf Case:} $2<q<3$. We take $\sigma=-\frac{1}{2}$, then $1+2\sigma=0$ and $q-2+2\sigma<0$.
 Since
$$\|\tilde{v}_n\|_{\dot{H}^1}=M_n^{-\frac{\sigma}{2}-1}\|v_n\|_{\dot{H}^1},$$
and $-\frac{\sigma}{2}-1=-\frac{3}{4}<0$.
So applying a similar argument in Case $3<q<6$, we may assume that
$$\tilde{u}_n\rightarrow \tilde{u},\quad \tilde{v}_n\rightarrow \tilde{v} \quad \text{in}~~ C_{loc}^{2}(\R^3),$$
where $(\tilde{u},\tilde{v})$ is a nontrivial and non-negative bounded  radial solution to
\begin{equation*}
\left\{
\begin{array}{ll}
-\Delta \tilde{u}=\tilde{v}\tilde{u}~~\hbox{in}~\R^3, \\
-\Delta \tilde{v}=\tilde{u}^2~~\hbox{in}~\R^3,
 \end{array}
\right.
\end{equation*}
and $\lim\limits_{|x|\to+\infty}\tilde{v}(x)=0$, also a contradiction to  Lemma \ref{lemma:20220601-l1}.
\end{proof}

Now, we define the set
\begin{equation*}
\mathcal{U}^{\Lambda_2}_{\Lambda_1}:= \{u\in H^1_r(\R^3) : u~\text{is~a~ground~state~solution~to}~\eqref{eq:20220420-e1}~\text{with}~\lambda\in [\Lambda_1, \Lambda_2]\},
\end{equation*}
 In view of Corollary \ref{2023411-cro1} and Proposition \ref{p2.2}, $\mathcal{U}^{\Lambda_2}_{\Lambda_1}\subset C_{r,0}(\R^3)$.

\begin{lemma}\label{20230413-l1} Let $0 < \Lambda_1 \leq \Lambda_2 < + \infty$. Then $\mathcal{U}_{\Lambda_1}^{\Lambda_2}$ is compact in $C_{r,0}(\R^3)$.
\end{lemma}
\begin{proof}  Note that a bounded set $\mathcal{A}\subset C_{r,0}(\R^3)$ is pre-compact if and only if $\mathcal{A}$ is equi-continuous on bounded sets and decay uniformly at infinity.
By a standard regularity argument we can check that the set $\mathcal{U}_{\Lambda_1}^{\Lambda_2}$ is bounded in $C^2(\R^3)$.

Now, we prove that $\mathcal{U}_{\Lambda_1}^{\Lambda_2}$ is uniform decay. We argue by contradiction and assume that there exists a $\varepsilon>0$, that can be assumed as arbitrarily small, and sequences $\{u_n\}\subset C_{r,0}(\R^3)$ and $r_n\rightarrow +\infty$ such that $u_n(r_n)\to\varepsilon$ and $u_n$ solves \eqref{eq:20220420-e1} with $\lambda=\lambda_n$. Put $$I(u):=\frac{1}{4\pi}\int_{\R^3}\frac{u^2}{|x|}dx=\int_0^{+\infty}u(s)^2sds$$ and $$K(r,s):= s^2(\frac{1}{s}-\frac{1}{r})\geq 0,~~0<s<r.$$
By Newton's theorem \cite[(9.7.5)]{LiebLoss.2001}, for radially symmetric functions $u_n$, we can conveniently express $v_n$ in polar coordinates as
\begin{equation*}
\aligned
v_n=\frac{1}{4\pi}\int_{\R^3}\frac{u_n(y)^2}{|x-y|}dy=\int_0^r u_n(s)^2\frac{s^2}{r}ds + \int_r^\infty u_n(s)^2 s ds
=I(u_n)-\int_0^r K(r,s)u_n(s)^2 ds.
\endaligned
\end{equation*}
Then, $u_n$ solves
\beq\label{eq:20220423-e1}
-\left(u''_n(r)+\frac{2}{r}u'_n(r)\right)=\left(I(u_n)-\lambda_n\right) u_n(r)+u_n(r)^{q-1}-\left(\int_0^r K(r,s)u_n(s)^2 ds\right)u_n(r).
\eeq
Put $\bar{u}_n(r):=u_n(r+r_n)$, then by (\ref{eq:20220423-e1}) we get
\begin{align*}
-\left(\bar{u}''_n+\frac{2}{r+r_n} \bar{u}'_n\right)=&\left(I(u_n)-\lambda_n\right) \bar{u}_n(r) +\bar{u}_n(r)^{q-1}\\
&-\left(\int_{0}^{r+r_n}K(r+r_n,s)u_n(s)^2 ds\right)\bar{u}_n(r), \quad  r>-r_n.
\end{align*}
Passing to subsequences (still denoted by $\lambda_n$ and $\bar{u}_n$) and then taking the limits, we get that $\lambda_n\rightarrow \lambda^*>0$ and that $\{u_n\}$ converges to $\bar{u}$. Noting that $$v_n(r+r_n)=I(u_n)-\int_{0}^{r+r_n}K(r+r_n,s)u_n(s)^2 ds\rightarrow 0~~\text{ as}~~ n\rightarrow +\infty,\quad\text{for~any~}r>0,$$
 we obtain that $\bar{u}$ is a  nontrivial solution  of the following equation
\begin{equation}\label{eq:20211204-e1}
-\bar{u}''=-\lambda^* \bar{u}+\bar{u}^{q-1}~\hbox{in}~\mathbb{R}
\end{equation}
 with $\bar{u}(0)=\varepsilon, \bar{u}\geq 0$. By Proposition \ref{p2.2}, $\bar{u}_n(r)$ is decreasing in $[-r_n,+\infty)$ and thus $\bar{u}$ is bounded and decreasing in $\mathbb{R}$. Hence, $\bar{u}(r)$ has a limit $\bar{u}_+$ at $r=+\infty$ and a limit $\bar{u}_-$ at $r=-\infty$. In particular, $0\leq \bar{u}_+\leq \bar{u}(r)\leq \bar{u}_-<+\infty, \forall r\in \mathbb{R}$ and $\bar{u}_-\geq \bar{u}(0)=\varepsilon>0$.
Here $\bar{u}_\pm$ satisfies
$$-\lambda^* \bar{u}_\pm+\bar{u}_\pm^{q-1}=0.$$
Since $\lambda >0$ is bounded away from $0$, we have that $\lambda^*>0$. Taking $\varepsilon >0$ smaller if necessary, we can assume that $\bar{u}_+^{q-2}<\lambda^*$, which implies that $\bar{u}_+=0$.
Put $f(s):=-\lambda^* s+s^{q-1}$ and $F(s):=\int_0^s f(t)dt$.
 Noting that $\displaystyle \lim_{t\rightarrow +\infty}\bar{u}'(t)=0$ and $\displaystyle \lim_{t\rightarrow +\infty}F(\bar{u}(t))=0$,
we have that
\beq\label{eq:20211204-be2}
\frac{1}{2}\bar{u}'(r)^2=\int_{r}^{+\infty} -\bar{u}''(t)\bar{u}'(t)dt
=\int_{r}^{+\infty}f(\bar{u}(t))\bar{u}'(t)dt
=-F(\bar{u}(r)), \forall r\in \mathbb{R}.
\eeq
By \cite[Theorem 5]{Berestycki1983}, there exist a unique solution $w$ (up to a translation) to the following equation
\beq\label{eq:20211204-e2}
-w''=f(w)~\hbox{in}~\mathbb{R}, w\in C^2(\mathbb{R}),\lim_{r \rightarrow \pm \infty}w(r)=0  ~\hbox{and $w(r_0)>0$ for some $r_0\in \mathbb{R}$}.
\eeq
Without loss of generality, we suppose that $w(0)=\max_{r\in \mathbb{R}}w(r)$, then
$$\begin{cases}
w(r)=w(-r);\\
w(r)>0,r\in \mathbb{R};\\
w(0)=\xi_0;\\
w'(r)<0,r>0,
\end{cases}$$
where $\xi_0>0$ is determined by
$$\xi_0:=\inf\left\{\xi>0:F(\xi)=0\right\},$$
see \cite[Theorem 5]{Berestycki1983} again.
By our choice of $\varepsilon >0$, we see that $f(s)<0, s\in (0,\varepsilon]$, and thus $\varepsilon<\xi_0$.
So there exists some $r_0>0$ such that $w(r_0)=\varepsilon$. Now, we let $\tilde{w}(r):=w(r+r_0)$, then
\beq\label{eq:20211204-e3}
-\tilde{w}''=f(\tilde{w}) ~\hbox{in}~\mathbb{R},~\tilde{w}(0)=\varepsilon.
\eeq
Furthermore, noting that $\displaystyle\lim_{r\rightarrow +\infty}\tilde{w}(r)=0$, applying a similar argument as that in \eqref{eq:20211204-be2}, we conclude that
\beq\label{eq:20211204-e4}
\tilde{w}'(r)=\begin{cases}
-\sqrt{-2F(\tilde{w}(r))}, \forall r\geq -r_0,\\
\sqrt{-2F(\tilde{w}(r))}, \forall r< -r_0.
\end{cases}
\eeq
Hence, both $\bar{u}$ and $\tilde{w}$ solve
\beq\label{eq:20211204-e5}
\begin{cases}
-u''(r)=f(u(r))~\hbox{in}~\mathbb{R},\\
u(0)=\varepsilon,\\
u'(0)=-\sqrt{-2F(\varepsilon)}.
\end{cases}
\eeq
By the uniqueness of solutions of initial value problem, we conclude $\bar{u}\equiv \tilde{w}$ in $\mathbb{R}$. Thus,
$$\bar{u}_-=\lim_{r\rightarrow -\infty}\bar{u}(r)=\lim_{r\rightarrow -\infty}\tilde{w}(r)=0,$$
a contradiction to $\bar{u}_{-} \geq\varepsilon>0$.
\end{proof}

\begin{lemma}\label{2023913-l5} Let $0 < \Lambda_1 \leq \Lambda_2 < + \infty$. Then the set $\mathcal{U}_{\Lambda_1}^{\Lambda_2}$ is compact in $H^1_r(\R^3)$.
\end{lemma}
\bp
By Lemma \ref{2023412-l1}, $\{u_n\}$ is bounded in $H^1(\R^3)$.
Now, for any sequence
$\{u_n\} \subset \mathcal{U}_{\Lambda_1}^{\Lambda_2} $, we may assume that $u_n \rightharpoonup u$ in $H^1(\R^3)$ and $\lambda_n \to \lambda^*\in [\Lambda_1, \Lambda_2]$.
By the continuity of the Riesz potential $I_2$, we have
$$(I_2\star u_n^2)u_n \rightharpoonup (I_2\star u^2)u~~\text{in}~H^1(\R^3).$$
In particular, $u$ is a positive radial solution to
\begin{equation}\label{20230413-e1}
-\Delta u+\lambda^* u=\left(I_2\star u^2\right)u+u^{q-1}.
\end{equation}
By compact embedding $H^1_r(\R^3)\hookrightarrow L^q(\R^3)$, we have
$$\int_{\R^3}|u_n|^q dx\to \int_{\R^3}|u|^q dx .$$
And it is standard to show that
$$\int_{\R^3}\left(I_2\star u_n^2\right)u_n^2 dx\to \int_{\R^3}\left(I_2\star u^2\right)u^2 dx .$$
Therefore, using equations \eqref{eq:20220420-e1} and \eqref{20230413-e1},
\begin{equation*}
\aligned
&\|\nabla u_n\|^2_{2}+\lambda_n\|u_n\|^2_{2}=\int_{\R^3}\left(I_2\star u_n^2\right)u_n^2 dx+\int_{\R^3}|u_n|^q dx  \\
\to &\int_{\R^3}\left(I_2\star u^2\right)u^2 dx+\int_{\R^3}|u|^q dx
=\|\nabla u\|^2_{2}+\lambda^*\|u\|^2_{2},
\endaligned
\end{equation*}
which implies that $u_n\to u$ in $H^1_r(\R^3)$. That is, $\mathcal{U}_{\Lambda_1}^{\Lambda_2}$ is compact in $H^1_r(\R^3)$.
\ep

\vskip4mm
\subsection{ Asymptotic behaviors of ground state solutions
for $\lambda\to 0^+$}

Let $u$ be a positive solution for \eqref{eq:20220420-e1}, we can see that the scaling
 function $\tilde{u}(x):=\lambda^{-\frac{1}{q-2}}u\left(\lambda^{-\frac{1}{2}}x\right)$ satisfies
\beq\label{2023416-e2}
-\Delta u+ u=\mu\left(I_2\star u^{2}\right)u+u^{q-1}
\eeq
with $\mu=\lambda^{-2\frac{q-3}{q-2}}$, and the scaling
 function $\bar{u}(x):=\lambda^{-1}u\left(\lambda^{-\frac{1}{2}}x\right)$ satisfies
 \beq\label{2023416-e1}
-\Delta u+ u=\left(I_2\star u^{2}\right)u+\nu u^{q-1}
\eeq
with $\nu=\lambda^{q-3}$, respectively.
S. Ma and V. Moroz \cite{ma2023asymptotic} obtained the following propositions.
\begin{proposition}\label{2023417-p1} (\cite[Theorem 2.5]{ma2023asymptotic}) Let $u_\mu$ be the radial ground state of
\eqref{2023416-e2},
 then for any sequence $\mu\to0^+$, there exists a subsequence such that
$u_\mu$ converges in $H^1(\R^3)$ to the solution $W\in H^1(\R^3)$ of the Schr\"{o}dinger equation \eqref{20220609-V1}.
\end{proposition}
\begin{proposition}\label{2023417-p2} (\cite[Theorem 2.4]{ma2023asymptotic}) Let $u_\nu$ be the radial ground state of \eqref{2023416-e1},
 then for any sequence $\nu\to0^+$, there exists a subsequence such that
$u_\nu$ converges in $H^1(\R^3)$ to the solution $U\in H^1(\R^3)$ of the Schr\"{o}dinger-Newton equation \eqref{20220609-VU3}.
\end{proposition}

\bl\label{lemma:20220422-l1}
Let $(u_n, v_n)\in H^1(\R^3)\times \dot{H}^1(\R^3)$ be positive ground state solutions to \eqref{eq:20220603-1} with $\lambda=\lambda_n\rightarrow 0^+$. Then
$$\limsup_{n\rightarrow +\infty}\|u_n\|_\infty=0\quad \text{and}\quad \limsup_{n\rightarrow +\infty}\|v_n\|_\infty=0.$$
\el
\bp
Obviously, $u_n$ satisfy
\beq\nonumber\label{eq:20220422-e2}
\begin{cases}
-\Delta u_n+\lambda_n u_n=v_nu_n+u_n^{q-1},~~\hbox{in}~\R^3,\\
-\Delta v_n=u_n^2.
\end{cases}
\eeq
By Lemma \ref{lemma:prior-estimate1}, we have that
\beq\nonumber\label{eq:20220422-e1}
\sup_{n\in \mathbb{N}}\|u_n\|_\infty<+\infty\quad \text{and}\quad \sup_{n\in \mathbb{N}}\left\|v_n\right\|_\infty<+\infty.
\eeq
Applying a standard elliptic estimate, we may assume that $u_n\rightarrow u$ and $v_n\rightarrow v$ in $C_{loc}^{2}(\R^3)$, where $(u,v)$ is a nonnegative radial decreasing function, which solves
\beq\label{eq:20220422-e3}
\begin{cases}
-\Delta u=vu+u^{q-1}~\hbox{in}~\R^3,\\
-\Delta v=u^2.
\end{cases}
\eeq
By Lemma \ref{20230122-l1}, $v_n$ is bounded in $\dot{H}^1(\R^3)$. Then
$$v_n \rightharpoonup \bar{v}~~ \text{in}~~\dot{H}^1(\R^3);\quad v_n(x)\to \bar{v}(x)~~ \text{a.e.}~~\R^3.$$
By the uniqueness of limits, we have $v=\bar{v}\in \dot{H}^1(\R^3)$ and thus $v=I_2\star u^2$.
Then, by \eqref{eq:20220422-e3}
$$-\Delta u=\left(I_2\star u^2\right)u+u^{q-1}\geq \left(I_2\star u^2\right)u.$$
From Lemma \ref{lemma:20220420-l1}, $u\equiv0$. Then by $-\Delta v=0$ and $v\in \dot{H}^1(\R^3)$, we get
$v\equiv0$.
Then, the conclusion follows.
\ep

More precisely, we have the following result.
\bl\label{lemma:20220423-bl1}
Let $(u_n, v_n)\in H^1(\R^3)\times \dot{H}^1(\R^3)$ be positive ground state solutions to \eqref{eq:20220603-1} with $\lambda=\lambda_n\rightarrow 0^+$. Let $M_n=\|u_n\|_\infty+\|v_n\|_\infty$. Then
\begin{itemize}
\item[(i)] If $2<q<3$, then
\beq\nonumber\label{eq:20220422-e4}
\limsup_{n\rightarrow +\infty}\frac{M_n^{q-2}}{\lambda_n}<+\infty.
\eeq
\item[(ii)] If $3<q<6$, then
\beq\nonumber\label{eq:20220422-e5}
\limsup_{n\rightarrow +\infty}\frac{M_n}{\lambda_n}<+\infty.
\eeq
\end{itemize}
\el
\bp
Similar to Lemma \ref{lemma:prior-estimate1},
letting $\tilde{u}_n(y)=\frac{1}{M_n}u_n(M_n^\sigma y)$ and $\tilde{v}_n(y)=\frac{1}{M_n}v_n(M_n^\sigma y)$, then
\begin{equation}\label{eq:20220603-i1}
\left\{
\begin{array}{ll}
\aligned
&-\Delta \tilde{u}_n=M_{n}^{1+2\sigma}\tilde{v}_n\tilde{u}_n+M_n^{q-2+2\sigma}\tilde{u}_n^{q-1}-\lambda_n M_{n}^{2\sigma}\tilde{u}_n \quad \text{in}~\R^3, \\
&-\Delta \tilde{v}_n=M_{n}^{1+2\sigma}\tilde{u}_n^2.
\endaligned
\end{array}
\right.
\end{equation}
Since $\|\tilde{u}_n\|_\infty+\|\tilde{v}_n\|_\infty=1$, then up to a subsequence, we assume that
$$\tilde{u}_n\rightarrow \tilde{u},\quad \tilde{v}_n\rightarrow \tilde{v} \quad \text{in}~~ C_{loc}^{2}(\R^3).$$
Note that $\lambda_n\to0$ and $M_n\rightarrow 0$ as $n\rightarrow +\infty$ due to Lemma \ref{lemma:20220422-l1}, we discuss the limit equation of \eqref{eq:20220603-i1} for the following two cases.

{\bf Case:} $2<q<3$.
Take $\sigma=-\frac{q-2}{2}$ in (\ref{eq:20220603-i1}).
If
$$\displaystyle \limsup_{n\rightarrow +\infty}\frac{M_n^{q-2}}{\lambda_n}=+\infty,$$
then by a similar argument of Lemma \ref{lemma:prior-estimate1} (Case $3<q<6$),
 $\tilde{u}$ is a nontrivial nonnegative solution to
\beq\label{eq:20220422-e12}
-\Delta \tilde{u}=\tilde{u}^{q-1}~~ \hbox{in}~\R^3.
\eeq
In this case, $q-1\in(1,2)$, then we get a contradiction by using Lemma \ref{lemma:20220603-l1}.
\par
{\bf Case:} $3<q<6$.
Take $\sigma=-\frac{1}{2}$ in (\ref{eq:20220603-i1}). If
$$\limsup_{n\rightarrow +\infty}\frac{M_n}{\lambda_n}=+\infty, $$
then $(\tilde{u}, \tilde{v})$ is a nontrivial nonnegative function satisfying
\begin{equation*}
\left\{
\begin{array}{ll}
-\Delta \tilde{u}=\tilde{v}\tilde{u}~~\hbox{in}~\R^3, \\
-\Delta \tilde{v}=\tilde{u}^2~~\hbox{in}~\R^3.
 \end{array}
\right.
\end{equation*}
Note that
$$\|\tilde{v}_n\|_{\dot{H}^1}=M_n^{-\frac{\sigma}{2}-1}\|v_n\|_{\dot{H}^1},$$
and $-\frac{\sigma}{2}-1=-\frac{3}{4}<0$.
So applying a similar argument in Lemma \ref{lemma:prior-estimate1} (Case $2<q<3$), we also get a contradiction to  Lemma \ref{lemma:20220601-l1}.
\ep

\bl\label{lemma:20220610-bl1}
Let $(u_n, v_n)\in H^1(\R^3)\times \dot{H}^1(\R^3)$ be positive ground state solutions to \eqref{eq:20220603-1} with $\lambda=\lambda_n\rightarrow 0^+$. Let $M_n=\|u_n\|_\infty+\|v_n\|_\infty$. Then
\begin{itemize}
\item[(i)] If $2<q<3$, then
\beq\nonumber\label{eq:20220610-e4}
0<\liminf_{n\rightarrow +\infty}\frac{M_n^{q-2}}{\lambda_n}.
\eeq
\item[(ii)] If $3<q<6$, then
\beq\nonumber\label{eq:20220610-e5}
0<\liminf_{n\rightarrow +\infty}\frac{M_n}{\lambda_n}.
\eeq
\end{itemize}
\el
\bp
Setting
$$\bar{u}_n(x):=\frac{1}{u_n(0)}u_n\left(\frac{x}{\sqrt{\lambda_n}}\right),$$
we have $\bar{u}_n(0)=\|\bar{u}_n\|_\infty=1$ and
\beq\label{eq:20220422-e6}
\begin{cases}
-\Delta \bar{u}_n+\bar{u}_n=\lambda_n^{-1}v_n\bar{u}_n+\lambda_n^{-1}u_n(0)^{q-2}\bar{u}_n^{q-1}~\hbox{in}~\R^3,\\
-\Delta v_n=u^2_n.
\end{cases}
\eeq
Taking $x=0$, by maximum principle and \eqref{eq:20220422-e6},
\beq\label{eq:20220422-e7}
1=\bar{u}_n(0)\leq -\Delta \bar{u}_n(0)+\bar{u}_n(0)\leq \left(\frac{v_n(0)}{\lambda_n}+\frac{u_n(0)^{q-2}}{\lambda_n}\right).
\eeq
{\bf (i)} $2<q<3$. Note that $v_n(0)\rightarrow 0$ as $n\rightarrow +\infty$ due to Lemma \ref{lemma:20220422-l1}.
Then by (\ref{eq:20220422-e7}), for $n$ large enough we have
$$\frac{v_n(0)^{q-2}}{\lambda_n}+\frac{u_n(0)^{q-2}}{\lambda_n}
\geq\frac{v_n(0)}{\lambda_n}+\frac{u_n(0)^{q-2}}{\lambda_n}\geq1.$$
Thus
$$
\liminf_{n\rightarrow +\infty}\frac{M_n^{q-2}}{\lambda_n}
=\liminf_{n\rightarrow +\infty}\frac{\left(u_n(0)+v_n(0)\right)^{q-2}}{\lambda_n}>0.
$$
\\
{\bf (ii)} $3<q<6$. Note that $u_n(0)\rightarrow 0$ as $n\rightarrow +\infty$ due to Lemma \ref{lemma:20220422-l1}.
Then by (\ref{eq:20220422-e7}), for $n$ large enough we have
$$\frac{v_n(0)}{\lambda_n}+\frac{u_n(0)}{\lambda_n}
\geq\frac{v_n(0)}{\lambda_n}+\frac{u_n(0)^{q-2}}{\lambda_n}\geq1.$$
Thus
$
\liminf_{n\rightarrow +\infty}\frac{M_n}{\lambda_n}
>0.
$
\ep

\bl\label{lemma:20220423-l1}
Let $(u_n, v_n)\in H^1(\R^3)\times \dot{H}^1(\R^3)$ be positive ground state solutions to \eqref{eq:20220603-1} with $\lambda=\lambda_n\rightarrow 0^+$. Let $M_n=\|u_n\|_\infty+\|v_n\|_\infty$.
Define
\beq\nonumber\label{eq:20210820-e4}
\bar{u}_n(x):=\frac{1}{M_n}u_n\left(\frac{x}{\sqrt{\lambda_n}}\right),
\quad\bar{v}_n(x):=\frac{1}{M_n}v_n\left(\frac{x}{\sqrt{\lambda_n}}\right)
\eeq
then $\bar{u}_n(x)\rightarrow 0$ as $\mid x\mid\rightarrow +\infty$ uniformly in $n\in \mathbb{N}$.
\el

\bp
By Lemmas \ref{lemma:20220423-bl1} and Lemma \ref{lemma:20220610-bl1}, up to a subsequence, we may assume that
$$\frac{M_n^{\eta}}{\lambda_n}\rightarrow C_0>0, \hbox{with}~\eta:=\min\{q-2,1\}.$$
$\bar{u}_n$ satisfies:
\beq\label{20220603:e1}
\begin{cases}
-\left({\bar{u}_n}''(r)+\frac{2}{r}{\bar{u}_n}'(r)\right)=-\bar{u}_n(r)+\frac{M_n^{q-2}\bar{u}_n(r)^{q-1}}{\lambda_n}+
\frac{M_n}{\lambda_n}\bar{v}_n\bar{u}_n(r),\\
-\left({\bar{v}_n}''(r)+\frac{2}{r}{\bar{v}_n}'(r)\right)=\frac{M_n}{\lambda_n}\bar{u}_n^2.
\end{cases}
\eeq

{\bf  For the case $2<q<3$,} $$\frac{M_n^{q-2}}{\lambda_n}\rightarrow C_0,\quad \frac{M_n}{\lambda_n}\rightarrow 0,\quad \text{as}~~n\to+\infty.$$
 We argue by contradiction and suppose there exists a sequence $r_n\rightarrow +\infty$ such that $\bar{u}_n(r_n)=\varepsilon$. By changing the origin to $r_n$ and passing to the limit (similar to the argument of Lemma \ref{20230413-l1}), from the first equation in \eqref{20220603:e1} we obtain a nontrivial solution $\bar{u}$ of the following equation
\begin{equation*}
-\bar{u}''=-\bar{u}+C_0\bar{u}^{q-1}, \;r\in \mathbb{R},
\end{equation*}
with $\bar{u}(0)=\varepsilon, \bar{u}\geq 0$ and bounded.  By Proposition \ref{p2.2}, we obtain that $\bar{u}$ is decreasing on $\mathbb{R}$. Hence, $\bar{u}$ has a limit $\bar{u}_+$ at $r=+\infty$ and a limit $\bar{u}_-$ at $r=-\infty$. In particular, $\bar{u}_\pm$ solve
$$-\bar{u}_\pm+C_0\bar{u}_{\pm}^{q-1}=0.$$
So by $\bar{u}_+<\varepsilon\leq \bar{u}_-$, we obtain that $\bar{u}_+=0$ and $\bar{u}_-=(C_0)^{\frac{1}{q-2}}$. Then, since from (\ref{eq:20210820-e4}), we have that $-\bar{u}''\leq 0$ on $\mathbb{R}$ necessarily $\bar{u}'(0) <0$ and using again that $-\bar{u}''\leq 0$ on $(- \infty, 0]$ we get a contradiction with the fact that $\bar{u}$ is bounded.

{\bf  For the case $3<q<6$,} $$\frac{M_n^{q-2}}{\lambda_n}\rightarrow 0,\quad \frac{M_n}{\lambda_n}\rightarrow C_0.$$
First, we show that $\bar{u}_n(x)\rightarrow 0$ as $\mid x\mid\rightarrow +\infty$ uniformly in $n\in \mathbb{N}$. We suppose there exists a sequence $r_n\rightarrow +\infty$ such that $\bar{u}_n(r_n)=\varepsilon$. By changing the origin to $r_n$ and passing to the limit of (\ref{20220603:e1}), we obtain a nontrivial solution $\bar{u}$ of the following equation,
\beq\label{eq:2023413-e3}
\begin{cases}
-\bar{u}''=-\bar{u}+C_0\bar{v}\bar{u}(r),\\
-\bar{v}''=C_0\bar{u}^2.
\end{cases}
\eeq
with $\bar{u}(0)=\varepsilon, \bar{u}\geq 0$ and bounded. By Proposition \ref{p2.2}, we obtain that $\bar{u}$ and $\bar{v}$ are decreasing on $\mathbb{R}$. Hence, $\bar{u}$ has a limit $\bar{u}_+$ at $r=+\infty$ and a limit $\bar{u}_-$ at $r=-\infty$.
Also, $\bar{v}$ has a limit $\bar{v}_+$ at $r=+\infty$ and a limit $\bar{v}_-$ at $r=-\infty$. In particular, $\bar{u}_\pm, \bar{v}_\pm$ solve
\beq\nonumber
\begin{cases}
0=-\bar{u}_\pm+C_0\bar{v}_\pm\bar{u}_\pm,\\
0=C_0\bar{u}_\pm^2.
\end{cases}
\eeq
So $\bar{u}_+=\bar{u}_-=0$, which contradicts with $\bar{u}(0)=\varepsilon>0$ and $-\bar{u}''\leq 0$ on $(- \infty, 0]$.
\ep

Now, we are ready to give the result about the behavior of positive solutions for $\lambda>0$ small.

\begin{theorem}\label{th4.3} (The behavior in the sense of $C_{r,0}(\R^3)$ as $\lambda\to0^+$)
Let  $(u_n, v_n)\in H^1(\R^3)\times \dot{H}^1(\R^3)$ be positive radial ground state solutions to \eqref{eq:20220603-1} with $\lambda=\lambda_n\rightarrow 0^+$.
\begin{itemize}
\item[(i)] If $2<q<3$,  define
$$
\tilde{u}_n(x):=\lambda_{n}^{-\frac{1}{q-2}}u_n\left(\frac{x}{\sqrt{\lambda_n}}\right),\quad \tilde{v}_n(x):=\lambda_{n}^{-\frac{1}{q-2}}v_n\left(\frac{x}{\sqrt{\lambda_n}}\right).
$$
Then, passing to a subsequence if necessary we have that $\tilde{u}_n\rightarrow W$ in $C_{r,0}(\R^3)$, where $W$ is the unique positive solution to \eqref{20220609-V1}.

\item[(ii)] If $3<q<6$, define
$$
\bar{u}_n(x):=\frac{1}{\lambda_{n}}u_n\left(\frac{x}{\sqrt{\lambda_n}}\right),\quad \bar{v}_n(x):=\frac{1}{\lambda_{n}}v_n\left(\frac{x}{\sqrt{\lambda_n}}\right).
$$
Passing to a subsequence if necessary we have that $\bar{u}_n\rightarrow U$ in $C_{r,0}(\R^3)$, where $U$ is the unique positive solution to \eqref{20220609-VU3}.
\end{itemize}
\end{theorem}
\begin{proof}
(i) Case $2<q<3$. By Lemma \ref{lemma:20220423-bl1} and Lemma \ref{lemma:20220610-bl1}, we know
\begin{equation}\label{eq:2022417-a1}
0<\liminf\limits_{n\to\infty}\frac{M_n^{q-2}}{\lambda_n}\leq \limsup\limits_{n\to\infty}\frac{M_n^{q-2}}{\lambda_n} < +\infty.
\end{equation}
Noting that $\tilde{u}_n, \tilde{v}_n$ satisfy
\begin{equation*}\label{eq:20220608-1}
\begin{cases}
-\Delta \tilde{u}_n+\tilde{u}_n=\lambda_n^{\frac{3-q}{q-2}}\tilde{v}_n\tilde{u}_n
+\tilde{u}_n^{q-1},\\
-\Delta \tilde{v}_n=\lambda_n^{\frac{3-q}{q-2}}\tilde{u}_n^2.
\end{cases}
\end{equation*}
By standard regularity argument, it is easy to see that $\tilde{u}_n, \tilde{v}_n$ are equi-continuous on
bounded sets. On the other hand, we remark that
$$\tilde{u}_n(x) = \frac{M_n}{\lambda_n^{\frac{1}{q-2}}} \bar{u}_n(x)~~\quad\text{and}~~\tilde{v}_n(x) = \frac{M_n}{\lambda_n^{\frac{1}{q-2}}} \bar{v}_n(x),$$
 where $\bar{u}_n(x)$ and $\bar{v}_n(x)$ are given in Lemma \ref{lemma:20220423-l1}. So, by Lemma \ref{lemma:20220423-l1} and \eqref{eq:2022417-a1}, we see that $\tilde{u}_n$ decay to $0$
uniformly at $\infty$. Hence, $\{\tilde{u}_n\}$ is pre-compact in $C_{r,0}(\R^3)$. And note that $\{\tilde{v}_n\}$ are bounded in $C_{r,0}(\R^3)$, passing to a subsequence
if necessary, we may assume that $\tilde{u}_n \to W \in C_{r,0}(\R^3)$.
Since $\lim\limits_{n\to\infty}\lambda_n^{\frac{3-q}{q-2}}\tilde{u}_n(x)=0$ for any $x\in\R^3$, then $W\in C_{r,0}(\R^3)$ solves
\eqref{20220609-V1} with $W(0) = \max\limits_{x\in\R^3}W(x) > 0.$

(ii) Case $3<q<6$. By Lemma \ref{lemma:20220423-bl1} and Lemma \ref{lemma:20220610-bl1}, we know
\begin{equation*}
0<\liminf\limits_{n\to\infty}\frac{M_n}{\lambda_n}\leq \limsup\limits_{n\to\infty}\frac{M_n}{\lambda_n} < +\infty.
\end{equation*}
Noting that $\bar{u}_n, \bar{v}_n$ satisfy
\beq\nonumber
\begin{cases}
-\Delta \bar{u}_n+\bar{u}_n=\bar{v}_n\bar{u}_n+\lambda_n^{q-3}\bar{u}_n^{q-1}~\hbox{in}~\R^3,\\
-\Delta \bar{v}_n=\bar{v}^2_n.
\end{cases}
\eeq

Using a similar argument of Case $2<q<3$, $\bar{u}_n \to U\in C_{r,0}(\R^3)$, and $U$ solves
\eqref{20220609-VU3}.
\end{proof}

\begin{theorem}\label{th:20220608} (The behavior in the sense of $H^1(\R^3)$ as $\lambda\to0^+$)
Let  $(u_n, v_n)\in H^1(\R^3)\times \dot{H}^1(\R^3)$ be positive radial ground state solutions to \eqref{eq:20220603-1} with $\lambda=\lambda_n\rightarrow 0^+$.
\begin{itemize}
\item[(i)] If $2<q<3$,  define
$$
\tilde{u}_n(x):=\lambda_{n}^{-\frac{1}{q-2}}u_n\left(\frac{x}{\sqrt{\lambda_n}}\right),\quad \tilde{v}_n(x):=\lambda_{n}^{-\frac{1}{q-2}}v_n\left(\frac{x}{\sqrt{\lambda_n}}\right).
$$
Passing to a subsequence if necessary we have that $\tilde{u}_n\rightarrow W$ in $H^1(\R^3)$, where $W$ is the unique positive solution to \eqref{20220609-V1}.

\item[(ii)] If $3<q<6$, define
$$\bar{u}_n(x):=\frac{1}{\lambda_{n}}u_n\left(\frac{x}{\sqrt{\lambda_n}}\right).$$
Passing to a subsequence if necessary we have that $\bar{u}_n\rightarrow U$ in $H^1(\R^3)$, where $U$ is the unique positive solution to \eqref{20220609-VU3}.
\end{itemize}
\end{theorem}
\bp
(i) $2<q<3$.
Since $\tilde{u}_n=\lambda_{n}^{-\frac{1}{q-2}}u_n\left(\frac{x}{\sqrt{\lambda_n}}\right)$ satisfies
\beq
-\Delta \tilde{u}_n+\tilde{u}_n=\lambda_n^{2\frac{3-q}{q-2}}\left(I_2\star\tilde{u}_n^2\right) \tilde{u}_n+\tilde{u}_n^{q-1},
\eeq
and $2\frac{3-q}{q-2}>0$, we get $\tilde{u}_n\rightarrow W$ in $H^1(\R^3)$ as $\lambda_n\to0^+$ by Proposition \ref{2023417-p1}
( \cite[Theorem 2.5.]{ma2023asymptotic}).

(ii) $3<q<6$.
Since $\bar{u}_n=\lambda_{n}^{-1}u_n\left(\frac{x}{\sqrt{\lambda_n}}\right)$ satisfies
\beq
-\Delta \bar{u}_n+\bar{u}_n=\left(I_2\star\bar{u}_n^2\right) \bar{u}_n+\lambda_n^{q-3}\bar{u}_n^{q-1},
\eeq
and $q-3>0$, we get $\bar{u}_n\rightarrow U$ in $H^1(\R^3)$ as $\lambda_n\to0^+$ by Proposition \ref{2023417-p2}
( \cite[Theorem 2.4.]{ma2023asymptotic}).
\ep

\vskip4mm
\subsection{ Asymptotic behaviors of positive solutions
for $\lambda\to +\infty$}

\begin{lemma}\label{lm:20220606-1} Let $(u_n, v_n)\in H^1(\R^3)\times \dot{H}^1(\R^3)$ be positive ground state solutions to \eqref{eq:20220603-1} with $\lambda=\lambda_n\rightarrow +\infty$. Let $M_n=\|u_n\|_\infty+\|v_n\|_\infty$. Then
\begin{equation}\nonumber\label{20220609-1}
\liminf\limits_{n\to+\infty} M_n  = +\infty,
\end{equation}
and
\begin{itemize}
\item[(i)] if $3<q<6$, then
\beq\label{eq:20220602-e4}
0<\liminf_{n\rightarrow +\infty}\frac{M_n^{q-2}}{\lambda_n};
\eeq
\item[(ii)] if $2<q<3$, then
\beq\label{eq:20220602-e5}
0<\liminf_{n\rightarrow +\infty}\frac{M_n}{\lambda_n}.
\eeq
\end{itemize}
\end{lemma}
\bp
Setting
$$\bar{u}_n(x):=\frac{1}{u_n(0)}u_n\left(\frac{x}{\sqrt{\lambda_n}}\right),$$
we have $\bar{u}_n(0)=\|\bar{u}_n\|_\infty=1$ and
\beq\label{eq:20220609-e6}
\begin{cases}
-\Delta \bar{u}_n+\bar{u}_n=\frac{v_n}{\lambda_n}\bar{u}_n+\frac{1}{\lambda_n }u_n(0)^{q-2}\bar{u}_n^{q-1}~\hbox{in}~\R^3,\\
-\Delta v_n=u^2_n.
\end{cases}
\eeq
Taking $x=0$, by maximum principle, it follows from \eqref{eq:20220609-e6} that
\begin{equation}\label{eq:20220610-e7}
\aligned
1=\bar{u}_n(0)\leq -\Delta \bar{u}_n(0)+\bar{u}_n(0)\leq \frac{v_n(0)+u_n(0)^{q-2}}{\lambda_n}.
\endaligned
\end{equation}
Then
\begin{equation}\label{eq:20220609-e7}
\aligned
1\leq \frac{v_n(0)+u_n(0)^{q-2}}{\lambda_n}
\leq \frac{M_n+M_n^{q-2}}{\lambda_n}.
\endaligned
\end{equation}
Since $\lambda_n \to +\infty$, we obtain $\liminf\limits_{n\to+\infty} M_n  = +\infty.$

(i) $q\in(3,6)$. By \eqref{eq:20220610-e7} we get
$$1<\liminf_{n\rightarrow +\infty}\frac{v_n(0)+u_n(0)^{q-2}}{\lambda_n}\leq \liminf_{n\rightarrow +\infty}\frac{M_n+M_n^{q-2}}{\lambda_n}.$$
Since $\liminf\limits_{n\to+\infty}M_n= +\infty$ and $q-2>1$, thus \eqref{eq:20220602-e4} also holds.

(ii) $q\in(2,3)$. By \eqref{eq:20220610-e7} we get
$$1<\liminf_{n\rightarrow +\infty}\frac{v_n(0)+u_n(0)^{q-2}}{\lambda_n}\leq \liminf_{n\rightarrow +\infty}\frac{M_n+M_n^{q-2}}{\lambda_n}.$$
Since $\liminf\limits_{n\to+\infty}M_n= +\infty$ and $q-2\leq1$, thus \eqref{eq:20220602-e5} also holds.
\ep

\begin{lemma}\label{lm:20220610-2} Let $(u_n, v_n)\in H^1(\R^3)\times \dot{H}^1(\R^3)$ be positive ground state solutions to \eqref{eq:20220603-1} with $\lambda=\lambda_n\rightarrow +\infty$. Let $M_n=\|u_n\|_\infty+\|v_n\|_\infty$. Then
\begin{itemize}
\item[(i)] if $3<q<6$, then
\beq\nonumber\label{eq:20220610-e4}
 \limsup_{n\rightarrow +\infty}\frac{M_n^{q-2}}{\lambda_n}<+\infty;
\eeq
\item[(ii)] if $2<q<3$, then
\beq\nonumber\label{eq:20220610-e5}
 \limsup_{n\rightarrow +\infty}\frac{M_n}{\lambda_n}<+\infty.
\eeq
\end{itemize}
\end{lemma}
\bp
Similar to Lemma \ref{lemma:prior-estimate1},
letting $\tilde{u}_n(y)=\frac{1}{M_n}u_n(M_n^\sigma y)$ and $\tilde{v}_n(y)=\frac{1}{M_n}v_n(M_n^\sigma y)$, then
\begin{equation}\label{eq:20220610-i1}
\left\{
\begin{array}{ll}
\aligned
&-\Delta \tilde{u}_n=M_{n}^{1+2\sigma}\tilde{v}_n\tilde{u}_n+M_n^{q-2+2\sigma}\tilde{u}_n^{q-1}-\lambda_n M_{n}^{2\sigma}\tilde{u}_n \quad \text{in}~\R^3, \\
&-\Delta \tilde{v}_n=M_{n}^{1+2\sigma}\tilde{u}_n^2,
\endaligned
\end{array}
\right.
\end{equation}
Note that $\lambda_n\to+\infty$ and $M_n\rightarrow +\infty$ as $n\rightarrow +\infty$ due to Lemma \ref{lm:20220606-1}.

{\bf Case:} $3<q<6$.
Take $\sigma=-\frac{q-2}{2}$ in (\ref{eq:20220610-i1}).
If
$$\displaystyle \limsup_{n\rightarrow +\infty}\frac{M_n^{q-2}}{\lambda_n}=+\infty,$$ then
by a similar argument of Lemma \ref{lemma:prior-estimate1},
up to a subsequence, $\tilde{u}_n\rightarrow \tilde{u}$ in $C_{loc}^{2}(\R^3)$, where $\tilde{u}$ is a nontrivial nonnegative solution to
\beq\nonumber\label{eq:20220610-e12}
-\Delta \tilde{u}=\tilde{u}^{q-1}~ \hbox{in}~\R^3.
\eeq
Since in this case $q-1\in(2,5)$, we get a contradiction by using Lemma \ref{lemma:20220603-l1}.

{\bf Case:} $2<q<3$.
Take $\sigma=-\frac{1}{2}$ in (\ref{eq:20220610-i1}). If $\limsup_{n\rightarrow +\infty}\frac{M_n}{\lambda_n}=+\infty$, then to a subsequence, $\tilde{u}_n\rightarrow \tilde{u}$ in $C_{loc}^{2}(\R^3)$, where $\tilde{u}$ is a nontrivial nonnegative function satisfying
\begin{equation*}
\left\{
\begin{array}{ll}
-\Delta \tilde{u}=\tilde{v}\tilde{u}~\hbox{in}~\R^3, \\
-\Delta \tilde{v}=\tilde{u}^2~\hbox{in}~\R^3,
 \end{array}
\right.
\end{equation*}
a contradiction to Lemma \ref{lemma:20220601-l1}.
\ep

\begin{theorem}\label{th5.3} (The behavior in the sense of $C_{r,0}(\R^3)$ as $\lambda\to +\infty$)
Let $(u_n, v_n)\in H^1(\R^3)\times \dot{H}^1(\R^3)$ be positive ground state solutions to \eqref{eq:20220603-1} with $\lambda=\lambda_n\rightarrow +\infty$.
\begin{itemize}
\item[(i)] If $3<q<6$, define
\beq\label{eq:20220601-e15}
\tilde{u}_n(x):=\lambda_{n}^{-\frac{1}{q-2}}u_n\left(\frac{x}{\sqrt{\lambda_n}}\right),\quad
\tilde{v}_n(x):=\lambda_{n}^{-\frac{1}{q-2}}v_n\left(\frac{x}{\sqrt{\lambda_n}}\right).
\eeq
Then, passing to a subsequence if necessary we have that $\tilde{u}_n\rightarrow W$ in $C_{r,0}(\R^3)$, where $W$ is the unique positive solution to
\eqref{20220609-V1}

\item[(ii)] If $2<q<3$, define
$
\bar{u}_n(x):=\frac{1}{\lambda_{n}}u_n\left(\frac{x}{\sqrt{\lambda_n}}\right),
$
then passing to a subsequence if necessary we have that $\bar{u}_n\rightarrow U$ in $C_{r,0}(\R^3)$, where $U$ is the unique positive solution to
\eqref{20220609-VU3}.
\end{itemize}
\end{theorem}
\bp
{\bf (i) $3<q<6$.} In this case, $\tilde{u}_n, \tilde{u}_n$ satisfy
\beq\nonumber
\begin{cases}
-\Delta \tilde{u}_n+\tilde{u}_n=\lambda_n^{-\frac{q-3}{q-2}}\tilde{v}_n\tilde{u}_n+\tilde{u}_n^{q-1},\\
-\Delta\tilde{v}_n=\lambda_n^{-\frac{q-3}{q-2}}\tilde{u}_n^2.
\end{cases}
\eeq
By Lemma \ref{lm:20220606-1} and Lemma \ref{lm:20220610-2}, we know
\begin{equation*}
0<\liminf\limits_{n\to\infty}\frac{M_n^{q-2}}{\lambda_n}\leq \limsup\limits_{n\to\infty}\frac{M_n^{q-2}}{\lambda_n} < +\infty.
\end{equation*}
Then by the same argument of Theorem \ref{th4.3}-(i) (Case $2<q<3$),
passing to a subsequence if necessary we have that $\tilde{u}_n\rightarrow W$ in $C_{r,0}(\R^3)$, where $W$ is the unique positive solution to \eqref{20220609-V1}.

{\bf (ii) $2<q<3$.} Using the same argument of Theorem \ref{th4.3}-(ii) (Case $3<q<6$), $\bar{u}_n \to U\in C_{r,0}(\R^3)$, and $U$ solves
\eqref{20220609-VU3}.
\ep

Now, similar to Theorem \ref{th:20220608} we also
have
\begin{theorem}\label{th5.4} (The behavior in the sense of $H^1(\R^3)$ as $\lambda\to +\infty$)
 Let  $\{u_n\}_{n=1}^{\infty}$ be positive radial solutions to \eqref{eq:20220420-e1} with $\lambda=\lambda_n\rightarrow +\infty$.
\begin{itemize}
\item[(i)] If $3<q<6$, define $\tilde{u}_n(x)$ and $\tilde{v}_n(x)$ as
\eqref{eq:20220601-e15}.
Then, passing to a subsequence if necessary we have that $\tilde{u}_n\rightarrow W$ in $H^1(\R^3)$.

\item[(ii)] If $2<q<3$, define
$
\bar{u}_n(x):=\frac{1}{\lambda_{n}}u_n\left(\frac{x}{\sqrt{\lambda_n}}\right).
$
Then, passing to a subsequence if necessary we have that $\bar{u}_n\rightarrow U$ in $H^1(\R^3)$.
\end{itemize}
\end{theorem}

\vskip4mm
{\section{ Uniqueness}}
\setcounter{equation}{0}

In this section, we prove the uniqueness of ground state solutions to \eqref{eq:20220603-1} provided
$\lambda > 0$ small or large enough.

First, we have the following results.
\begin{proposition}\label{p2.6} (\cite[Proposition 2.7]{louis}) Let $L_+$ be the linearized operator arising from the ground state solution $W$ of (\ref{20220609-V1}),
\beq\label{eq:20230415-e2}
L_+ \left( \xi \right)  =-\Delta\xi + \xi - (q-1)W^{q-2}\xi.
\eeq
Then $L_+$ has a null kernel in $H^1_r(\R^3)$.
\end{proposition}
We also need the uniqueness and nondegeneracy results for the Schr\"{o}dinger-Newton equation
\beq\label{2023415-e111}
  \left\{
\begin{array}{ll}
\aligned
&-\Delta u+u=vu~~\hbox{in}~\R^3,\\
&-\Delta v= u^2~~\hbox{in}~\R^3,
\endaligned
\end{array}
\right.
\eeq
which is equivalent to (\ref{20220609-VU3}).
\begin{proposition}\label{p2.7}  Let $\mathcal{L}_+$ be the linearized operator arising from the ground state solution $(U, V)$ for (\ref{2023415-e111}),
\beq\nonumber
\mathcal{L}_+ \left(\begin{matrix} \xi\\ \zeta \end{matrix} \right)  =\left(\begin{matrix} -\Delta\xi + \xi- V \xi - \zeta U \\ -\Delta \zeta- 2\xi U\end{matrix} \right).
\eeq
Then $\mathcal{L}_+$ has a null
kernel in $H^1_r(\R^3)\times H^1_r(\R^3)$.
\end{proposition}
\bp
By the nondegeneracy result in \cite{Lenzmann2009,MR1677740,MR2492602}, we have
$$\text{Ker}\mathcal{L}_+=\text{span}\{(\partial_i U, \partial_i V), i = 1,2, 3\}.$$
Since $\partial_i U, \partial_i V$ are non-radial symmetric function, thus $\mathcal{L}_+$ has a null
kernel in $H^1_r(\R^3)\times H^1_r(\R^3)$.
\ep

Now, we prove Theorem \ref{th1.2}.
\begin{proof}
(i) We first consider the case where $\lambda> 0$ is small. We argue by contradiction
and suppose there exist two families of positive solutions $(u^{(1)}_\lambda, v^{(1)}_\lambda)$
and $(u^{(2)}_\lambda, v^{(2)}_\lambda)$ to \eqref{eq:20220603-1} with $\lambda\to 0^+$.
\par
Case $q\in(2,3)$. Let
$$\tilde{u}^{(i)}_\lambda(x):=\lambda^{-\frac{1}{q-2}}u^{(i)}_\lambda(\lambda^{-\frac{1}{2}}x),\quad \tilde{v}^{(i)}_\lambda(x):=\lambda^{-\frac{1}{q-2}}v^{(i)}_\lambda(\lambda^{-\frac{1}{2}}x),\quad i=1,2.$$
Then $(\tilde{u}^{(i)}_\lambda, \tilde{v}^{(i)}_\lambda)\in H^1_r(\R^3)\times\dot{H}^1(\R^3)$ $(i=1,2)$ are two families of positive radial solutions to
\beq\nonumber
\left\{
\begin{array}{ll}
\aligned
&-\Delta u+u=\lambda^{-\frac{q-3}{q-2}}vu+u^{q-1}~~\text{in}~~\R^3,\\
&-\Delta v=\lambda^{-\frac{q-3}{q-2}}u^2.
\endaligned
\end{array}
\right.
\eeq
By Theorem \ref{th4.3} and Theorem \ref{th:20220608}, one has
$$\tilde{u}^{(i)}_\lambda (x) \to W~\text{as}~\lambda\to 0^+~\text{both~in}~C_{r,0}(\R^3)~\text{and~in}~H^1(\R^3),~~ i = 1, 2.$$
Define
$$\xi_\lambda:=\frac{\tilde{u}^{(1)}_\lambda-\tilde{u}^{(2)}_\lambda}
{\|\tilde{u}^{(1)}_\lambda-\tilde{u}^{(2)}_\lambda\|_\infty+\|\tilde{v}^{(1)}_\lambda-\tilde{v}^{(2)}_\lambda\|_\infty},
\quad \zeta_\lambda:=\frac{\tilde{v}^{(1)}_\lambda-\tilde{v}^{(2)}_\lambda}
{\|\tilde{u}^{(1)}_\lambda-\tilde{u}^{(2)}_\lambda\|_\infty+\|\tilde{v}^{(1)}_\lambda-\tilde{v}^{(2)}_\lambda\|_\infty}.$$
Then $\|\xi_\lambda\|_\infty+\|\zeta_\lambda\|_\infty=1$.
By mean value theorem, for any $x\in\R^3$, there exists some
$\theta(x)\in [0, 1]$ such that
\begin{equation*}
\aligned
(\tilde{u}^{(1)}_\lambda)^{q-1}-(\tilde{u}^{(2)}_\lambda)^{q-1}=(q-1)\left(\theta(x)\tilde{u}^{(1)}_\lambda+(1-\theta(x))\tilde{u}^{(2)}_\lambda\right)^{q-2}(\tilde{u}^{(1)}_\lambda-\tilde{u}^{(2)}_\lambda).
 \endaligned
\end{equation*}
Then by
\begin{equation*}
\aligned
\tilde{v}^{(1)}_\lambda\tilde{u}^{(1)}_\lambda-\tilde{v}^{(2)}_\lambda\tilde{u}^{(2)}_\lambda
=\tilde{v}^{(1)}_\lambda(\tilde{u}^{(1)}_\lambda-\tilde{u}^{(2)}_\lambda)
+(\tilde{v}^{(1)}_\lambda-\tilde{v}^{(2)}_\lambda)\tilde{u}^{(2)}_\lambda,
 \endaligned
\end{equation*}
we have
\begin{equation}\label{5.1}
\left\{
\begin{array}{ll}
\aligned
&-\Delta\xi_\lambda = -\xi_\lambda +\lambda^{-\frac{q-3}{q-2}}
\tilde{v}^{(1)}_\lambda\xi_\lambda+\lambda^{-\frac{q-3}{q-2}}\zeta_\lambda\tilde{u}^{(2)}_\lambda+(q-1)\left(\theta(x)\tilde{u}^{(1)}_\lambda+(1-\theta(x))\tilde{u}^{(2)}_\lambda\right)^{q-2}\xi_\lambda,\\
&-\Delta \zeta_\lambda=\lambda^{-\frac{q-3}{q-2}}\left(\tilde{u}^{(1)}_\lambda+\tilde{u}^{(2)}_\lambda\right)\xi_\lambda.
\endaligned
\end{array}
\right.
\end{equation}
By Lemma \ref{lemma:20220423-bl1}, for $i=1,2$ $$\left\|\tilde{u}^{(i)}_\lambda\right\|_\infty=\left\|\frac{u^{(i)}_\lambda(\lambda^{-\frac{1}{2}}x)}{\lambda^{\frac{1}{q-2}}}\right\|_\infty
~~\text{and}~~ \left\|\tilde{v}^{(i)}_\lambda\right\|_\infty=\left\|\frac{v^{(i)}_\lambda(\lambda^{-\frac{1}{2}}x)}{\lambda^{\frac{1}{q-2}}}\right\|_\infty$$
are uniformly bounded as $\lambda\to0^+$.
Then by the facts that $\|\xi_\lambda\|_\infty, \|\zeta_\lambda\|_\infty \leq 1$, $\theta(x)\in [0, 1]$ and $\tilde{u}^{(i)}_\lambda\to W$ in $C_{r,0}(\R^3)$, one can see that the right hand side of (\ref{5.1}) is in $L^\infty(\R^3)$. Hence, passing
to a subsequence if necessary, we can assume that
$$\xi_\lambda \to\xi,~\zeta_\lambda \to\zeta,~~\text{in}~ C^2_{\text{loc}}(\R^3),$$
where $\xi$ is a
radial bounded function satisfying
\begin{equation}\nonumber
\left\{
\begin{array}{ll}
\aligned
&-\Delta\xi + \xi =(q-1) W^{q-2}\xi,\\
&-\Delta \zeta=0.
\endaligned
\end{array}
\right.
\end{equation}
Then $\|\xi\|_\infty= 1$ and $\|\zeta\|_\infty= 0$. Standard elliptic estimates imply that $\xi$ is a strong solution. Then
by the decay of $W$ and applying a comparison principle, we obtain that $\xi$ is exponentially decaying to $0$ as $|x|\to\infty$. Hence, $\xi\in C_{r,0}(\R^3)\cap H^1_r(\R^3).$ At this point,
Proposition \ref{p2.6} provides a contradiction.
\par
Case $q\in(3,6)$: Let
$$\bar{u}^{(i)}_\lambda:=\lambda^{-1}u^{(i)}_\lambda(\lambda^{-\frac{1}{2}}x),~~ \bar{v}^{(i)}_\lambda:=\lambda^{-1}v^{(i)}_\lambda(\lambda^{-\frac{1}{2}}x),\quad i=1,2.$$
Then $(\bar{u}^{(i)}_\lambda, \bar{v}^{(i)}_\lambda)\in H^1_r(\R^3)\times \dot{H}^1_r(\R^3)$ $(i=1,2)$ are two families of positive radial solutions to
\beq\nonumber
\left\{
\begin{array}{ll}
\aligned
&-\Delta u+u=vu+\lambda^{q-3}u^{q-1}~~\text{in}~~\R^3,\\
&-\Delta v=u^2.
\endaligned
\end{array}
\right.
\eeq
By Theorem \ref{th4.3} and Theorem \ref{th:20220608}, one has
$$\bar{u}^{(i)}_\lambda (x) \to U~\text{as}~\lambda\to 0^+~\text{both~in}~C_{r,0}(\R^3)~\text{and~in}~H^1(\R^3),~~ i = 1, 2.$$
By Lemma \ref{lemma:20220610-bl1}-(ii), $\bar{v}^{(i)}_\lambda$ is bounded in $L^\infty(\R^3)$, then up to a subsequence,
$$\bar{v}^{(i)}_\lambda (x) \to V~\text{as}~\lambda\to 0^+~\text{both~in}~C_{loc}^2(\R^3),~~ i = 1, 2.$$
We study the normalization
$$\xi_\lambda:=\frac{\bar{u}^{(1)}_\lambda-\bar{u}^{(2)}_\lambda}
{\|\bar{u}^{(1)}_\lambda-\bar{u}^{(2)}_\lambda\|_\infty+\|\bar{v}^{(1)}_\lambda-\bar{v}^{(2)}_\lambda\|_\infty}
~~\text{and}~~\zeta_\lambda:=\frac{\bar{v}^{(1)}_\lambda-\bar{v}^{(2)}_\lambda}
{\|\bar{u}^{(1)}_\lambda-\bar{u}^{(2)}_\lambda\|_\infty+\|\bar{v}^{(1)}_\lambda-\bar{v}^{(2)}_\lambda\|_\infty}.$$
Similar to Case $q\in(2,3)$, passing
to a subsequence if necessary, we can assume that
$$(\xi_\lambda,\zeta_\lambda) \to(\xi,\zeta)~~\text{in}~ C^2_{\text{loc}}(\R^3)\times C^2_{\text{loc}}(\R^3),$$
where $(\xi,\zeta)$ is
radial bounded function satisfying
\beq\nonumber
\left\{
\begin{array}{ll}
\aligned
&-\Delta \xi+\xi=V\xi+U\zeta~~\text{in}~~\R^3,\\
&-\Delta \zeta=2U\xi.
\endaligned
\end{array}
\right.
\eeq
Since $\|\xi\|_\infty+\|\zeta\|_\infty= 1$, standard elliptic estimates imply that $(\xi,\zeta)$ is a strong solution. Then by the decay of $U$ and applying a comparison principle, we obtain that $\xi$ is exponentially decaying to $0$ as $|x|\to\infty$. Hence, $\xi\in C_{r,0}(\R^3)\cap H^1_r(\R^3).$ At this point,
Proposition \ref{p2.7} provides a contradiction.

(ii) Now we consider the case where $\lambda> 0$ is large.
\par
Case $q\in(3,6)$: the proof of uniqueness is similar to the case $q\in(2,3)$ in (i).
\par
Case $q\in(2,3)$: the proof of uniqueness is similar to the case $q\in(3,6)$ in (i).
\end{proof}

\vskip4mm
{\section{Nondegeneracy}}
 \setcounter{equation}{0}

{\subsection{ Decomposition into spherical harmonics  }

In this section, we assume that $u_{\lambda}\in\mathcal{M}_\lambda$ and we prove
that it is nondegenerate for $\lambda$ sufficiently close to $0$ or $+\infty$, where $\mathcal{M}_\lambda$ is the set of nontrivial solutions defined in \eqref{20230417-e2}. For this, we denote by $\perp_{H^1}$ and $\perp_{\dot{H}^1}$ the
orthogonality relation in $H^1(\R^3)$ and $\dot{H}^1(\R^3)$ respectively.

By \cite{Lenzmann2009}, for linearized operators $L_+$ arising from ground states $W$ for NLS with local nonlinearities,
it is a well-known fact that $Ker{L_+} = \{0\}$ when $L_+$ is restricted to radial functions implies that $Ker{L_+}$ is
spanned by $\{\partial_i W\}_{i=1}^3$.

The proof, however, involves some Sturm-Liouville theory which is not applicable to $\mathcal{L}_\lambda^+$ given in \eqref{eq:20220913-e2},
due to the presence of the nonlocal term. Also, recall that Newton's theorem \cite[(9.7.5)]{LiebLoss.2001} is not at our disposal, since we do not restrict ourselves to radial functions anymore. To overcome this difficulty, we have
to develop Perron-Frobenius-type arguments for the action of $\mathcal{L}_\lambda^+$ with respect to decomposition into spherical harmonics.

Now, let $u_{\lambda_n}$ be positive ground state solution for equation \eqref{eq:20220420-e1} with $\lambda=\lambda_n$. Then
 $\tilde{u}_n(x)=\lambda_{n}^{-\frac{1}{q-2}}u_n(x/\sqrt{\lambda_n}):=u_{\mu_n}(x)$ satisfy
 \beq\nonumber
-\Delta u+ u=\mu_n\left(I_2\star u^{2}\right)u+ u^{q-1} \quad \text{in}~\R^3,
\eeq
and $\bar{u}_n(x)=\lambda_n^{-1}u_n(x/\sqrt{\lambda_n}):=u_{\nu_n}(x)$ satisfy
 \beq\nonumber
-\Delta u+ u= \left(I_2\star u^{2}\right)u+ \nu_n u^{q-1} \quad \text{in}~\R^3.
\eeq
where $\mu_n=\lambda_n^{2\frac{3-q}{q-2}} $, $\nu_n=\lambda_n^{q-3}$.

Recall that from Theorem \ref{th:20220608} and Theorem \ref{th5.4}, we have
\begin{itemize}
\item[(i)] if $2<q<3$, $u_{\mu_n}\to W$ in $H^1(\R^3)$, as $\lambda_n\to0^+$;
\item[(ii)] if $3<q<6$, $u_{\nu_n}\to U$ in $H^1(\R^3)$, as $\lambda_n\to0^+$;
\item[(iii)] if $3<q<6$, $u_{\mu_n}\to W$ in $H^1(\R^3)$, as $\lambda_n\to+\infty$;
\item[(iv)] if $2<q<3$, $u_{\nu_n}\to U$ in $H^1(\R^3)$, as $\lambda_n\to+\infty$,
\end{itemize}
where $W$ is the unique positive solution for the Schr\"{o}dinger equation \eqref{20220609-V1},
$U$ is the unique positive solution for the Schr\"{o}dinger-Newton equation \eqref{20220609-VU3}.

Let $(U_\nu, V_\nu) = (u_\nu(|x|), v_\nu(|x|))$ be a ground state for
\beq\label{eq:20230412-e2}
  \left\{
\begin{array}{ll}
\aligned
&-\Delta u+ u=2 vu+\nu u^{q-1}~~\hbox{in}~\R^3,\\
&-\Delta v= u^2~~\hbox{in}~\R^3.
\endaligned
\end{array}
\right.
\eeq
And Let $(U_\mu, V_\mu)$ be the positive ground state for
\beq\label{eq:20230415-e12}
  \left\{
\begin{array}{ll}
\aligned
&-\Delta u+ u=2\mu vu+ u^{q-1}~~\hbox{in}~\R^3,\\
&-\Delta v= \mu u^2~~\hbox{in}~\R^3.
\endaligned
\end{array}
\right.
\eeq
To show the nondegeneracy of the ground state solution $(U_\lambda, V_\lambda)$ for \eqref{eq:20220603-1} as $\lambda$ close to $0$ or $+\infty$, it is suffice to prove the nondegeneracy of $(U_\nu, V_\nu)$ and $(U_\mu, V_\mu)$ as $\nu$ and $\mu$ close to $0$, respectively.
For this, from now on, we will use the uniqueness and nondegeneracy results for the Schr\"{o}dinger equation \eqref{20220609-V1} and the Schr\"{o}dinger-Newton equation \eqref{20220609-VU3}.
 Namely, we recall that there exists a unique radial ground state $W$ for \eqref{20220609-V1} such that
\begin{equation*}\label{eq:20221118-2}
Ker(L_+)=span\{\partial_jW,j=1,2,3\},
\end{equation*}
where the linear operator $L_+$ associated to $W$ is defined by \eqref{eq:20230415-e2}.
And we also recall from \cite{Lieb1976/77,Ma2010,MR1677740,MR2492602} that there exists a unique radial ground state $(U, V)$ for
\beq\label{20221223-e111}
  \left\{
\begin{array}{ll}
\aligned
&-\Delta u+u=2 vu~~\hbox{in}~\R^3,\\
&-\Delta v= u^2
\endaligned
\end{array}
\right.
\eeq
such that
\beq\label{eq:20221118-2}
Ker(\mathcal{L}_+)=span\{(\partial_jU,\partial_jV),j=1,2,3\},
\eeq
where the linear operator $\mathcal{L}_+$ associated to $(U, V)$ is defined by
  \beq\label{eq:20230415-e3}
\mathcal{L}_+ \left(\begin{matrix} \xi\\ \zeta \end{matrix} \right)  =\left(\begin{matrix} -\Delta\xi + \xi- 2V \xi - 2\zeta U \\ -\Delta \zeta- 2\xi U\end{matrix} \right).
\eeq

\begin{remark}
Previously, we use $(U,V)$ to denote the unique radial ground state solution of the Schr\"{o}dinger-Newton equation \eqref{2023415-e111}.
Indeed, the unique radial ground state solution solution for system \eqref{20221223-e111} is $(\frac{1}{\sqrt{2}}U, V)$, system \eqref{2023415-e111} is equivalent to \eqref{20221223-e111} in the scaling sense. We also write the unique solution for system \eqref{20221223-e111} as $(U,V)$. We use the system \eqref{20221223-e111} instead of system \eqref{2023415-e111} to simplify the representation of the energy functional.
\end{remark}

 Define the linear operator $\mathcal{L}_{\nu}^+$ associated to $(U_{\nu}, V_{\nu})$ by
  \beq\nonumber
\mathcal{L}_{\nu}^+ \left(\begin{matrix} \xi\\ \zeta \end{matrix} \right)  =\left(\begin{matrix} -\Delta\xi + \xi- 2V_{\nu} \xi - 2\zeta U_{\nu}-\nu(q-1)U_\nu^{q-2}\xi \\ -\Delta \zeta- 2\xi U_{\nu}\end{matrix} \right).
\eeq
Define the energy functional $I_{\nu}: H^1(\R^3)\times \dot{H}^1(\R^3) \mapsto \R$ for \eqref{eq:20230412-e2} as
\beq\nonumber
I_{\nu}(u,v):=\frac{1}{2}\|u\|_{H^1}^2+\frac{1}{2}\|v\|_{\dot{H}^1}^2-\int_{\R^3}u^2 v dx-\frac{\nu}{q}\int_{\R^3}u^q dx.
\eeq
Moreover, for any $\varphi, \psi \in H^1(\R^3)$,
$$ \langle I'_{\nu}(u,v),(\varphi,\psi)\rangle=(u,\varphi)_{H^1}+(v,\psi)_{\dot{H}^1}-2\int_{\R^3}uv\varphi dx-\int_{\R^3}u^2\psi dx-\nu\int_{\R^3}u^{q-1}\varphi dx.
$$
And the second order Gateaux derivative
$I''_{\nu} (u_{\nu}, v_{\nu} )$ possess the following property.
\bl\label{3.71} For every $\varphi \perp_{H^1} u_{\nu}$ and $\psi\perp_{\dot{H}^1} v_{\nu}$ we have that
\beq\nonumber
\aligned
0 \leq &I''_{\nu} (u_{\nu}, v_{\nu} )[(\varphi,\psi), (\varphi,\psi)] \\
= &\|\varphi\|^2_{H^1}+\|\psi\|_{\dot{H}^1}^2-2\int_{\R^3} v_{\nu}\varphi^2 dx-4\int_{\R^3} u_{\nu}\varphi \psi dx-\nu(q-1)\int_{\R^3} u_\nu^{q-2}\varphi^2 dx.
\endaligned
\eeq
\el
\bp
Let $\varepsilon>0$. Since $\varphi \perp_{H^1} u_{\nu}$ and $\psi\perp_{\dot{H}^{1}} v_{\nu}$, we have
\beq\label{eq:20221120-12}
\|\varepsilon \varphi+u_{\nu}\|^2_{H^1}={\varepsilon}^2 \|\varphi\|^2_{H^1} +\|u_{\nu}\|^2_{H^1}, \quad \|\varepsilon \psi+v_{\nu}\|^2_{\dot{H}^{1}}={\varepsilon}^2 \|\psi\|^2_{\dot{H}^{1}} +\|v_{\nu}\|^2_{\dot{H}^{1}}.
\eeq
Moreover, by using the system \eqref{eq:20230412-e2}, we have
$$\int_{\R^3} 2v_{\nu} u_{\nu} \varphi+\nu u_\nu^{q-1}\varphi dx =0,\quad \int_{\R^3} u_{\nu}^2 \psi dx=0,$$
and thus by Taylor expansion,
\beq\label{eq:20221120-21}
\aligned
&\int_{\R^3}(\varepsilon\psi+v_{\nu}) (\varepsilon\varphi+u_{\nu})^2dx+\frac{\nu}{q}\int_{\R^3}|\varepsilon \varphi+u_{\nu}|^{q} dx\\
=&\int_{\R^3} v_{\nu} u_{\nu}^2 dx+\varepsilon^3\int_{\R^3} \varphi^2\psi dx +2\varepsilon^2\int_{\R^3}\psi\varphi u_{\nu}dx  \\
&+\varepsilon^2 \int_{\R^3}\varphi^2 v_{\nu} dx+\varepsilon\int_{\R^3} u_{\nu}^2\psi dx +2\varepsilon\int_{\R^3} u_{\nu}\varphi v_{\nu} dx\\
&+\frac{\nu}{q}\int_{\R^3} u_{\nu}^{q}dx+\varepsilon\nu\int_{\R^3} u_\nu^{q-1}\varphi dx+\frac{(q-1)\nu\varepsilon^2}{2}\int_{\R^3} u_\nu^{q-2}\varphi^2 dx +o(\varepsilon^2)\\
=&\int_{\R^3} v_{\nu} u_{\nu}^2 dx+\varepsilon^3\int_{\R^3} \varphi^2\psi dx+
2\varepsilon^2\int_{\R^3}\psi\varphi u_{\nu} dx+\varepsilon^2 \int_{\R^3}\varphi^2 v_{\nu} dx+o(\varepsilon^2).
\endaligned
\eeq
From \eqref{eq:20221120-21} and \eqref{eq:20221120-12} we obtain
\begin{align*}
&\frac{1}{2}\|\varepsilon \varphi+u_{\nu}\|^2_{H^1}+\frac{1}{2}\|\varepsilon \psi+v_{\nu}\|^2_{\dot{H}^{1}}-\int_{\R^3}(\varepsilon\psi+v_{\nu})|\varepsilon \varphi+u_{\nu}|^{2}dx-\frac{\nu}{q}\int_{\R^3}|\varepsilon \varphi+u_{\nu}|^{q}dx \\
=&\frac{1}{2}\|u_{\nu}\|^2_{H^1}+\frac{1}{2}\|v_{\nu}\|^2_{\dot{H}^{1}}-\int_{\R^3} v_{\nu} u_{\nu}^2 dx-\frac{\nu}{q}\int_{\R^3} u_{\nu}^{q}dx \\
&+\frac{\varepsilon^2}{2}\left(\| \varphi\|^2_{H^1}+ \|\psi\|^2_{\dot{H}^{1}}-4\int_{\R^3}\psi\varphi u_{\nu} dx-2 \int_{\R^3}\varphi^2 v_{\nu} dx-(q-1)\nu\int_{\R^3} u_\nu^{q-2}\varphi^2 dx\right)+o(\varepsilon^2).
\notag
\end{align*}
Then the desired result follows since the ground state $(u_{\nu}, v_{\nu})$ attains the minimal of $I_\nu(u,v)$.
\ep

\begin{corollary}\label{3.8} For any $(h,l)\in H^1(\R_{+};r^{2})\times H^1(\R_+; r^{2})$
\beq\label{eq:20221120-4e}
\aligned
	&A_1((h,l),(h,l))\\
:=&\int_{\R_{+}} h^2_r   r^{2}dr +  2\int_{\R_{+}} h^2  dr  +\int_{\R_+}h^2 r^{2}dr \\
	&+ \int_{\R_+} l_r^2 r^{2} dr+2\int_{\R_+}l^2 dr-2  \int_{\R_+}v_{\nu} h^2 r^{2}dr-4  \int_{\R_+}h l u_{\nu}  r^{2}dr \\
&-\nu(q-1)\int_{\R_+} u_{\nu}^{q-2}h^2 r^{2}dr \\
\geq &0.
\endaligned
\eeq
\end{corollary}
\bp
Let $h\in H^1(\R_{+}; r^{2})$, $l\in H^1(\R_+; r^{2})$ and define
$$\Phi_i(x):=h(|x|)\frac{x^i}{|x|}, \quad \Psi_i(x):=l(|x|)\frac{x^i}{|x|},\quad i=1,2,3.$$
By a direct computation,
\beq\label{eq:20221221-1}
\aligned
\sum_{i}|\nabla \Phi_i|^2=h_r^2+2\frac{h^2}{r^2}, \quad\sum_{i}|\nabla \Psi_i|^2=l_r^2+2\frac{l^2}{r^2},\quad
\sum_{i} \Phi_i \Psi_i=h_r l_r.
\endaligned
\eeq
Since $u_{\nu}$ and $v_{\nu}$ are radial, then by odd symmetry we have
 $$\int_{\R^3}\nabla\Phi_i \nabla u_{\nu} dx=-\int_{\R^3}\Phi_i \Delta u_{\nu} dx=0 ,\quad \int_{\R^3}\Phi_i u_{\nu} dx=0 ,$$
 $$\int_{\R^3}\nabla\Psi_i \nabla v_{\nu} dx=-\int_{\R^3}\Psi_i \Delta v_{\nu} dx=0,$$
 and so $\Phi_i\perp_{H^1} u_{\nu}, \Psi_i\perp_{\dot{H}^1} v_{\nu}$. Then Lemma \ref{3.71} and \eqref{eq:20221221-1} yield \eqref{eq:20221120-4e}.
\ep

Let $\theta= \frac{x}{|x|} \in \mathbb{S}^2$, the unit sphere in $\R^3$. Let $\Delta_r$ be
the Laplacian operator in radial coordinates and $\Delta_{\mathbb{S}^2}$ the Laplacian-Beltrami operator. We recall that
$$\Delta u = \Delta_ru + \frac{1}{r^2}\Delta_{\mathbb{S}^2}u,$$
and we consider the spherical harmonics on $\R^3$, i.e., the solution of the classical eigenvalue problem
$$-\Delta_{\mathbb{S}^{2}}Y^i_k =\lambda_k Y^i_k~~~\text{on}~\mathbb{S}^{2},\quad k\in\mathbb{N}.$$
Let $n_k$ be the multiplicity of $\lambda_k$.
\begin{proposition}\label{20221223-e1}(\cite{GH}) The eigenvalue $\lambda_k=k(k + 1)$ for $k\in\mathbb{N}$.
$$n_0=1,~~ Y_0=Const;\quad \quad n_1=3, ~~Y^i_1=\frac{x^i}{|x|}~\text{for}~i=1,2,3,$$
and
\begin{equation*}
\langle Y_k^i , Y_k^j \rangle_{L^2(\mathbb{S}^{2})} =
\left\{
\begin{array}{ll}
\aligned
&1, ~~~\text{if}~i = j;\\
&0, ~~~\text{if}~ i \neq j.
\endaligned
\end{array}
\right.
\end{equation*}
\end{proposition}

\bl\label{3.9} Let $(\varphi, \psi)\in \text{Ker}( I''_{\nu} (u_{\nu}, v_{\nu} ))$. Then
$$\varphi = \varphi_0(|x|) +\sum_{i=1}^3 c^i\partial_i u_{\nu},\quad
\psi = \psi_0(|x|) +\sum_{i=1}^3 c^i\partial_i v_{\nu}, $$
where
$\varphi_0(r ) = \int_{\mathbb{S}^{2}}\varphi(r\theta)d\sigma(\theta)$, $\psi_0(r ) = \int_{\mathbb{S}^{2}}\psi(r\theta)d\sigma(\theta)$
and $c^i\in\R.$
\el
\bp
Let $(\varphi,\psi) \in Ker(I''_{\nu}(u_{\nu},v_{\nu}))$ which means
\beq\label{20221219-e1}
    \left\{
\begin{array}{ll}
  \aligned
  & -\Delta  \varphi +\varphi=2 \psi u_{\nu} +2 v_{\nu} \varphi+\nu(q-1)u_{\nu}^{q-2}\varphi,\\
  &-\Delta \psi=2 u_{\nu} \varphi.
  \endaligned
\end{array}
\right.
\eeq
For any $(\Psi, \Phi) \in H^1(\R^3)\times H^1(\R^3)$, we have
\beq\label{eq:20221120-7}
    \left\{
\begin{array}{ll}
\aligned
&      \int_{\R^3} \nabla \varphi  \cdot \nabla\Psi dx +\int_{\R^3} \varphi\Psi dx =2\int_{\R^3}(\psi u_{\nu}+v_{\nu}\varphi)\Psi dx+\nu(q-1)\int_{\R^3} u_{\nu}^{q-2}\varphi \Psi dx,\\
&\int_{\R^{N}} \nabla\psi \cdot \nabla\Phi dx=2\int_{\R^3} u_{\nu} \varphi\Phi dx.
\endaligned
\end{array}
\right.
\eeq
 Now we decompose $ \varphi $, $\psi$ in the spherical harmonics and we obtain
\beq\label{eq:20221120-8}
 \varphi (x)=\sum \limits_{k \in \N} \sum \limits_{i=1}^{n_k} f^k_i(r)Y^i_k (\theta),\quad \psi(x)=\sum \limits_{k \in \N} \sum \limits_{i=1}^{n_k} g^k_i(r)Y^i_k (\theta),
\eeq
where $f^k_i \in H^1(\R_{+};r^{2})$, $g^k_i \in H^1(\R_+; r^{2})$, $r=|x|$ and $\theta=\frac{x}{|x|}$.
By testing the first equation in \eqref{eq:20221120-7} against the function $\Psi=h(|x|)Y^i_k$ and using polar coordinates and Proposition \ref{20221223-e1}, we obtain that, for any $h \in H^1(\R_{+}; r^{2})$, any $k \in\mathbb{N}$ and any $i \in \left[1,n_k\right]$,
\begin{equation}\label{20221219-ep1}
\aligned
&A_k((f^k_i,g^k_i), h)_1 \\
:=& \int_{\R_{+}} (f^k_i)_r h_r  r^{2}dr+\lambda_k \int_{\R_{+}} f^k_i h dr + \int_{\R_+} f^k_i hr^{2}dr \\
&-2  \int_{\R_+}u_{\nu} g^k_i  h  r^{2}dr-2  \int_{\R_+} v_{\nu} f^k_i hr^{2}dr-\nu(q-1)\int_{\R_+} u_\nu^{q-2} f^k_i hr^{2}dr\\
=&0.
\endaligned
\end{equation}

By testing the second equation in  \eqref{eq:20221120-7} against the function $\Phi=l(|x|)Y^i_k$ and using polar coordinates and Proposition \ref{20221223-e1}, we obtain that, for any $l\in H^1(\R_+; r^{2})$, any $k \in\mathbb{N}$ and any $i \in \left[1,n_k\right]$,
\begin{equation}\label{20221221-ep1}
\aligned
&A_k((f^k_i,g^k_i),l)_2 \\
:=&\int_{\R_+} (g^k_i)_r l_r r^{2}dr+\lambda_k \int_{\R_+} g^k_i l dr-2 \int_{\R_+} u_{\nu} f^k_i l r^{2}dr\\
=&0.
\endaligned
\end{equation}

Let
$$A_k((f^k_i,g^k_i),(h,l))
:=A_k((f^k_i,g^k_i),h)_1+A_k((f^k_i,g^k_i),l)_2,$$  take $h=f^k_i$ and $l=g^k_i$, we observe that
\begin{equation}\nonumber
\aligned
&A_k((f^k_i,g^k_i),(f^k_i,g^k_i)) \\
=&\int_{\R_{+}} |(f^k_i)_r |^2  r^{2}dr +\lambda_k \int_{\R_{+}} |f^k_i|^2 dr + \int_{\R_+} |f^k_i|^2r^{2}dr \\
&-4  \int_{\R_+}f^k_i g^k_i u_{\nu}   r^{2}dr-2  \int_{\R_+} v_{\nu} |f^k_i|^2r^{2}dr-\nu(q-1)\int_{\R_+} u_\nu^{q-2} |f^k_i|^2r^{2}dr\\
&+\int_{\R_+} |(g^k_i)_r|^2 r^{2}dr+\lambda_k \int_{\R_+} |g^k_i|^2 dr \\
=&A_1((f^k_i,g^k_i),(f^k_i,g^k_i))+(\lambda_k - 2 ) \int_{\R_{+}} |f^k_i|^2  dr +(\lambda_k - 2 ) \int_{\R_{+}} |g^k_i|^2  dr\\
=&0,
\endaligned
\end{equation}
where $A_1$ is defined in \eqref{eq:20221120-4e}.
By Corollary \ref{3.8} ($A_1((f^k_i,g^k_i),(f^k_i,g^k_i))\geq0$) and the fact that the eigenvalue $\lambda_k>2$ for $k \geq 2$, we obtain from the identities above that
\begin{align*}
0&=A_k((f^k_i,g^k_i),(f^k_i,g^k_i)) \notag \\
&\geq(\lambda_k - 2 ) \int_{\R_{+}} |f^k_i|^2     dr +(\lambda_k - 2 ) \int_{\R_{+}}|g^k_i|^2     dr.\notag
\end{align*}
As a consequence, $f^k_i=0$ for every $k \geq 2$. Accordingly, \eqref{eq:20221120-8} becomes
$$ \varphi (x)= \sum \limits_{i=1}^{3} f^1_i(|x|)Y^i_1 (\frac{x}{|x|}),\quad \psi(x)= \sum \limits_{i=1}^{3} g^1_i(|x|)Y^i_1 (\frac{x}{|x|}).$$
Here, by Proposition \ref{20221223-e1},
$$Y^i_1(\frac{x}{|x|})=\frac{x^i}{|x|}=\theta^i.$$
And we have from the orthogonality of $Y^i_1$ in $L^2(\mathbb{S}^2)$ that
$$f^1_i(r)= \int_{\mathbb{S}^{2}}  \varphi (r\theta){\theta}^i d\sigma(\theta),~~g^1_i(r)= \int_{\mathbb{S}^{2}} \psi(r\theta){\theta}^i d\sigma(\theta).$$
To complete the proof we need to characterize $f^1_i$ and $g^1_i$. For this, we notice that, for $i=1,2,3$, $f^1_i(t,0)=0$, $g^1_i(0)=0$,
\begin{equation}\label{eq:20221120-9}
\aligned
&A_1((f^1_i,g^1_i),h)_1 \\
=&\int_{\R_{+}} (f^1_i)_r h_r  r^{2}dr+ 2  \int_{\R_{+}} f^1_i h dr + \int_{\R_+} f^1_i hr^{2}dr \\
&-2  \int_{\R_+}u_{\nu} g^1_i  h  r^{2}dr-2  \int_{\R_+} v_{\nu} f^1_i hr^{2}dr-\nu(q-1)\int_{\R_+} u_\nu^{q-2}f_i^1 h r^2 dr \\
=&0,
\endaligned
\end{equation}
and
\begin{equation}\label{eq:20221120-92}
\aligned
&A_1((f^1_i,g^1_i),l)_2 \\
=&\int_{\R_+} (g^1_i)_r l_r r^{2}dr+ 2 \int_{\R_+} g^1_i l dr-2\int_{\R_+} u_{\nu} f^1_i l r^{2}dr\\
=&0,
\endaligned
\end{equation}
for every $h \in H^1(\R_+;r^{2})$ and $l\in H^1(\R_+; r^{2})$, due to the eigenvalue $\lambda_1=2$ and \eqref{20221219-ep1}-\eqref{20221221-ep1}.

Now we define $\bar{U}(|x|)=u_{\nu}(x)$ and $\bar{V}(|x|)=v_{\nu}(x)$.
Then we have
\beq\nonumber
\left\{
\begin{array}{ll}
	\aligned
	& -\partial_{rr}\bar{U}-\frac{2}{r}\partial_r \bar{U}+ \bar{U}= 2 \bar{U}\bar{V}+\nu\bar{U}^{q-1}~~~\text{on}~\R_+, \\
&-\partial_{rr}\bar{V}-\frac{2}{r}\partial_r \bar{V}=\bar{U}^2~~~\text{on}~\R_+, \\
	&\lim\limits_{r\searrow 0} r^{2} \bar{U}_r=0,\quad \lim\limits_{r\searrow 0} r^{2} \bar{V}_r=0.
	\endaligned
\end{array}
\right.
\eeq
We differentiating the above equation with respect to $r$. We obtain
\beq\label{eq:20221120-10}
\left\{
\begin{array}{ll}
	\aligned
	&-\partial_r\left(\frac{1}{r^{2}}\partial_r(r^{2}\bar{U}_r)\right)+ \bar{U}_r=2  (\bar{U}_r\bar{V}+\bar{U}\bar{V}_r)+\nu(q-1)\bar{U}^{q-2}\bar{U}_r~~~\text{on}~\R_+, \\
&-\partial_r\left(\frac{1}{r^{2}}\partial_r(r^{2}\bar{V}_r)\right)=2\bar{U}\bar{U}_r~~~\text{on}~\R_+, \\
	&\lim\limits_{r\searrow 0} r^{2} \bar{U}_r=0,\quad \lim\limits_{r\searrow 0} r^{2} \bar{V}_r=0.
	\endaligned
\end{array}
\right.
\eeq

By Proposition \ref{p2.2}, $\bar{U}, \bar{V}$ are positive,
radially symmetric and decreasing, we may assume that $\bar{U}_r, \bar{V}_r<0$ on $\R_{+}$.

Given $f \in C^\infty_c (\R_{+})$, by testing the first equation of \eqref{eq:20221120-10} with $\frac{f^2}{\bar{U}_{r}}r^2$,
\beq\label{20221220-ea3}
\aligned
&-\int_{\R_+} f^2r^{2} dr+2\int_{\R_+} (\bar{V}+\bar{U}\bar{V}_r/\bar{U}_{r})f^2r^{2} dr+\nu(q-1)\int_{\R_+}\bar{U}^{q-2}f^2 r^2dr\\
=&-\int_{\R_+}\partial_r\left(\frac{1}{r^{2}}\partial_r(r^{2}\bar{U}_r)\right)\frac{f^2}{\bar{U}_r}r^2dr \\
:=&I
\endaligned
\eeq
Integrating by parts, we get
\beq\nonumber
\aligned
I=& 2 \int_{\R_+}\frac{1}{r^{2}}\partial_r(r^{2}\bar{U}_r) r\frac{f^2}{\bar{U}_r}dr+\int_{\R_+} \frac{1}{r^{2}}\partial_r(r^{2}\bar{U}_r) r^{2} \partial_r \left(\frac{f^2}{\bar{U}_r}\right) dr \\
=&- 2 \int_{\R_+}\left(r \frac{2ff_r\bar{U}_r-f^2\bar{U}_{rr}}{\bar{U}_r}-f^2 \right) dr+\int_{\R_+}\left( 2 r\bar{U}_r+r^{2}\bar{U}_{rr}\right)\frac{2ff_r\bar{U}_r-f^2\bar{U}_{rr}}{\bar{U}_r^2} dr\\
=& 2 \int_{\R_+} f^2      dr-\int_{\R_+} \left[\left(\frac{\bar{U}_{rr}}{\bar{U}_r}f\right)^2-2\frac{\bar{U}_{rr}}{\bar{U}_r}ff_r\right] r^{2} dr .
\endaligned
\eeq
Then by \eqref{20221220-ea3}, we get
\begin{equation}\label{202212201813}
\aligned
&\int_{\R_{+}}f_r^2  r^{2}dr+ 2 \int_{\R_{+}}f^2 dr+\int_{\R_+}f^2 r^{2}dr \\
&-2\int_{\R_+}(\bar{V}+\bar{U}\bar{V}_r/\bar{U}_{r}) f^2 r^{2}dr-\nu(q-1)\int_{\R_+}\bar{U}^{q-2}f^2 r^2dr  \\
=&\int_{\R_{+}}f_r^2  r^{2}dr+\int_{\R_+} \left[\left(\frac{\bar{U}_{rr}}{\bar{U}_r}f\right)^2-2\frac{\bar{U}_{rr}}{\bar{U}_r}ff_r\right] r^{2} dr .
\endaligned
\end{equation}
Note from \eqref{eq:20221120-9} that
\begin{equation}\label{202212201812}
\aligned
&A_1((f,g),f)_1\\
=& \int_{\R_{+}} f^2_r   r^{2}dr+ 2  \int_{\R_{+}} f^2 dr + \int_{\R_+} f^2 r^{2}dr \\
&-2  \int_{\R_+} fg \bar{U}  r^{2}dr-2  \int_{\R_+} \bar{V} f^2 r^{2}dr-\nu(q-1)\int_{\R_+}\bar{U}^{q-2}f^2 r^2 dr.
\endaligned
\end{equation}
By \eqref{202212201813} and \eqref{202212201812}, we get
\begin{equation}\label{202212201814}
\aligned
&A_1((f,g),f)_1 \\
= &\int_{\R_{+}}\left[f_r^2  r^{2}+ \left(\frac{\bar{U}_{rr}}{\bar{U}_r}f\right)^2-2\frac{\bar{U}_{rr}}{\bar{U}_r}ff_r r^{2}\right] dr \\
&-2  \int_{\R_+} fg \bar{U}  r^{2}dr + 2\int_{\R_+}\frac{\bar{U}\bar{V}_r}{\bar{U}_{r}} f^2 r^{2}dr\\
=&\int_{\R_{+}}|\bar{U}_{r}\nabla (f/ \bar{U}_r)|^2  r^{2}dr-2  \int_{\R_+} fg \bar{U}  r^{2}dr + 2\int_{\R_+}\frac{\bar{U}\bar{V}_r}{\bar{U}_{r}} f^2 r^{2}dr.
\endaligned
\end{equation}

Given $g \in C^\infty_c (\R^{+})$,
by testing the third equation of \eqref{eq:20221120-10} with $\frac{g^2}{\bar{V}_{r}}r^{2}$, we have
\beq\nonumber
\aligned
2\int_{\R_+}\bar{U}\frac{\bar{U}_r}{\bar{V}_r} g^2 r^{2} dr
=-\int_{\R_+}\partial_r\left(\frac{1}{r^{2}}\partial_r(r^{2}\bar{V}_r)\right)r^{2}\frac{g^2}{\bar{V}_r}dr
:=II.
\endaligned
\eeq
Integrating by parts, similar to I, we get
\beq\nonumber
\aligned
II= 2 \int_{\R_+} g^2      dr-\int_{\R_+} \left[\left(\frac{\bar{V}_{rr}}{\bar{V}_r}g\right)^2-2\frac{\bar{V}_{rr}}{\bar{V}_r}gg_r\right] r^{2} dr .
\endaligned
\eeq
\allowbreak
Therefore,
\beq\label{202212201803}
\aligned
2\int_{\R_+}\bar{U}\frac{\bar{U}_r}{\bar{V}_r} g^2 r^{2} dr
= 2 \int_{\R_+} g^2      dr-\int_{\R_+} \left[\left(\frac{\bar{V}_{rr}}{\bar{V}_r}g\right)^2-2\frac{\bar{V}_{rr}}{\bar{V}_r}gg_r\right] r^{2} dr .
\endaligned
\eeq

Note from \eqref{eq:20221120-92} that
\begin{equation}\label{202212201804}
\aligned
A_1((f,g),g)_2=\int_{\R_+} g_r^2 r^{2}dr+ 2 \int_{\R_+} g^2     dr-2 \int_{\R_+} \bar{U} fg r^{2}dr.
\endaligned
\end{equation}
By \eqref{202212201803} and \eqref{202212201804}, we get
\begin{equation}\label{202212201805}
\aligned
&A_1((f,g), g)_2
=2\int_{\R_+}\bar{U}\frac{\bar{U}_r}{\bar{V}_r} g^2 r^{2} dr-2\int_{\R_+} \bar{U} fg r^{2}dr \\
&+\int_{\R_+} \left[g_r^2+\left(\frac{\bar{V}_{rr}}{\bar{V}_r}g\right)^2-2\frac{\bar{V}_{rr}}{\bar{V}_r}gg_r\right] r^{2} dr.
\endaligned
\end{equation}
Combining \eqref{202212201805} and \eqref{202212201814}, we get
\begin{equation}\label{202212201827}
\aligned
&A_1((f,g),(f,g))=A_1((f,g),f)_1+A_1((f,g), g)_2 \\
=&\int_{\R_{+}}|\bar{U}_{r}\nabla (f/ \bar{U}_r)|^2  r^{2}dr +2  \int_{\R_+}\left(\frac{\bar{V}_r}{\bar{U}_r}f^2+\frac{\bar{U}_r}{\bar{V}_r}g^2 -2fg  \right) \bar{U} r^{2}dr\\
&+\int_{\R_+} \left[g_r^2+\left(\frac{\bar{V}_{rr}}{\bar{V}_r}g\right)^2-2\frac{\bar{V}_{rr}}{\bar{V}_r}gg_r\right] r^{2} dr .
\endaligned
\end{equation}
Therefore,
we obtain
\begin{equation}\label{202212201828}
\aligned
&A_1((f,g),(f,g)) \geq \int_{\R_{+}}|\bar{U}_{r}\nabla (f/ \bar{U}_r)|^2  r^{2}dr \\
+&2  \int_{\R_+}\left(\sqrt{\frac{\bar{V}_r}{\bar{U}_r}}f-\sqrt{\frac{\bar{U}_r}{\bar{V}_r}}g\right)^2   \bar{U} r^{2}dr.
\endaligned
\end{equation}
In particular, by density we have that, for every $i=1,2,3$,
\begin{equation}\nonumber
\aligned
0=&A_1((f^1_i,g^1_i),(f^1_i,g^1_i))\geq \int_{\R_{+}} \left|\bar{U}_{r}\nabla\left(\frac{f^1_i}{ \bar{U}_{r}}\right)\right|^2  r^{2}dr
\\
&+2  \int_{\R_+}\left(\sqrt{\frac{\bar{V}_r}{\bar{U}_r}}f^1_i-\sqrt{\frac{\bar{U}_r}{\bar{V}_r}}g^1_i\right)^2   \bar{U} r^{2}dr.
\endaligned
\end{equation}
This implies that the last two terms vanish and therefore
$$\frac{f^1_i}{\bar{U}_{r}}=\frac{g^1_i}{\bar{V}_{r}} \equiv c^i$$
for some constant $c^i \in \R$. We then conclude that
$$f^1_i(|x|)=c^i \partial_r\bar{U}(|x|),\quad g^1_i(|x|)=c^i \partial_r\bar{V}(|x|) \quad\forall x \in \R^3.$$
Thus, we have proved that for any $(\varphi,\psi) \in Ker(I''_{\nu}(u_{\nu},v_{\nu}))$
$$ \varphi (x)=\varphi(x)=f^0_1(|x|)+\sum \limits_{i=1}^{3} f^1_i(|x|) \frac{x^i}{|x|}=f^0_1(|x|)+\sum \limits_{i=1}^{3} c^i \partial_i u_{\nu}(x),$$
and
$$\psi(x)=g^0_1(|x|)+\sum \limits_{i=1}^{3} g^1_i(|x|) \frac{x^i}{|x|}=g^0_1(|x|)+\sum \limits_{i=1}^{3} c^i \partial_i v_{\nu}(x),$$
as desired.
\ep

Now, define the energy functional $I_{\mu}: H^1(\R^3)\times \dot{H}^1(\R^3) \mapsto \R$ for \eqref{eq:20230415-e12} as
\beq\nonumber
I_{\mu}(u,v):=\frac{1}{2}\|u\|_{H^1}^2+\frac{1}{2}\|v\|_{\dot{H}^1}^2-\mu\int_{\R^3}u^2 v dx-\frac{1}{q}\int_{\R^3}u^q dx.
\eeq
\bl\label{3.19} Let $(\varphi, \psi)\in \text{Ker}( I''_{\mu} (u_{\mu}, v_{\mu} ))$. Then
$$\varphi = \varphi_0(|x|) +\sum_{i=1}^3 c^i\partial_i u_{\mu},\quad
\psi = \psi_0(|x|) +\sum_{i=1}^3 c^i\partial_i v_{\mu}, $$
where
$\varphi_0(r ) = \int_{\mathbb{S}^{2}}\varphi(r\theta)d\sigma(\theta)$, $\psi_0(r ) = \int_{\mathbb{S}^{2}}\psi(r\theta)d\sigma(\theta)$
and $c^i\in\R.$
\el
\bp
The proof is similar to Lemma \ref{3.9}, we omit it.
\ep
Now we are ready to prove our nondegeneracy result for $\nu$ (resp. $\mu$) close to $0^+$.

{\subsection{  Completion of the proof of Theorem \ref{th1.3}. }
 Let $(w_{\nu}, \vartheta_{\nu})  \in Ker(I''_{\nu}(u_{\nu}, v_{\nu} ))$ and $(w_{\mu}, \vartheta_{\mu})  \in Ker(I''_{\mu}(u_{\mu}, v_{\mu} ))$ be radial
functions. To proof Theorem \ref{th1.3}, according to Lemma \ref{3.9} and Lemma \ref{3.19}, it is suffice to prove the following Claim.
\\
Claim 1: If $\nu$ is close to $0^+$, we have $w_{\nu} =\vartheta_{\nu} \equiv 0$;\\
Claim 2: If $\mu$ is close to $0^+$, we have $w_{\mu} =\vartheta_{\mu} \equiv 0$.

Indeed, if we obtain Claim 1, then Theorem \ref{th1.3} holds in cases (ii) and (iv):
\begin{itemize}
\item[(ii)] $3<q<6$, $\lambda$ close to $0$;
\item[(iv)] $2<q<3$, $\lambda$ close to $+\infty$;
\end{itemize}
if we obtain Claim 2, then Theorem \ref{th1.3} holds in cases (i) and (iii):
\begin{itemize}
\item[(i)] $2<q<3$, $\lambda$ close to $0$;
\item[(iii)] $3<q<6$, $\lambda$ close to $+\infty$.
\end{itemize}
We only prove Claim 1, since the proof of Claim 2 is the same with Claim 1.

Recall that $\nu=\lambda^{q-3}$, and by Theorem \ref{th:20220608} and Theorem \ref{th5.4},
 $$u_{\nu}\to U~~\text{and}~~ v_{\nu}\to V,\quad\text{in}~~H^1(\R^3),\quad \text{as}~\nu\to0^+.$$

Assume by contradiction that there exists a sequence $\nu_n$ still denoted by $\nu$ with $\nu\to 0^+$ and such that $(w_{\nu}, \vartheta_{\nu})\neq (0,0)$. Up to normalization, we can assume that $\|w_{\nu}\|^2_{H^1}=\|\vartheta_{\nu}\|^2_{H^1}=1$, and up to a subsequence,
$$w_{\nu} \rightharpoonup w~~\text{and}~~\vartheta_{\nu} \rightharpoonup \vartheta,\quad\text{in}~~H^1(\R^3),\quad \text{as}~\nu\to0^+.$$
Then by the uniform decaying property of $u_{\nu}$, for any $\varphi, \phi \in C^\infty_c (\R^3)$ we have
\beq\nonumber
\aligned
 &\int_{\R^3} u_{\nu}\vartheta_{\nu}\varphi dx \to \int_{\R^3} U\vartheta\varphi dx,\quad
 \int_{\R^3} w_{\nu}v_{\nu}\varphi dx\to\int_{\R^3}  w V\varphi dx, \\
 & \nu(q-1)\int_{\R^3}u_{\nu}^{q-2}w_{\nu}\varphi dx\to0, \quad\int_{\R^3} w_{\nu} u_{\nu}\phi dx\to\int_{\R^3}  w U\phi dx,
\endaligned
\eeq
as $\nu\to0^+$.
 Next we observe that $(w_{\nu},\vartheta_{\nu})$ is a solution of the linearized equation and therefore for any $\varphi, \phi \in C^\infty_c (\R^3)$
\beq\nonumber
  \left\{
\begin{array}{ll}
\aligned
&-\int_{\R^3} w_{\nu} \Delta  \varphi dx+\int_{\R^3} w_{\nu} \varphi dx=2\int_{\R^3}w_{\nu}v_{\nu}\varphi dx+2\int_{\R^3}u_{\nu}\vartheta_{\nu} \varphi dx+\nu(q-1)\int_{\R^3}u_{\nu}^{q-2}w_{\nu}\varphi dx, \\
&-\int_{\R^3} \vartheta_{\nu} \Delta \phi dx=2\int_{\R^3}w_{\nu} u_{\nu} \phi dx,
\endaligned
\end{array}
\right.
\eeq
 we infer that
\beq\nonumber
  \left\{
\begin{array}{ll}
\aligned
&-\int_{\R^3} w \Delta \varphi dx+\int_{\R^3} w\varphi dx=2\int_{\R^3}w V\varphi dx+2\int_{\R^3}U \vartheta \varphi dx, \\
&-\int_{\R^3} \vartheta \Delta \phi dx=2\int_{\R^3}w U \phi dx.
\endaligned
\end{array}
\right.
\eeq
We then conclude that $(w,\vartheta)$ is radial, nontrivial and belongs to $Ker(\mathcal{L}_+)$. This is clearly a contradiction and the claim is proved.

\vskip4mm
%%%%%%%%%%%%%%%%%%%%


\begin{thebibliography}{10}

\bibitem{1} {\sc T. Akahori, S. Ibrahim, N. Ikoma, H. Kikuchi and H. Nawa}, {\em Uniqueness and nondegeneracy of ground states to nonlinear scalar field equations involving the Sobolev critical exponent in their nonlinearities for high frequencies}, Calc. Var. Partial Differential Equations, 58 (2019), Paper No. 120, 32 pp.

\bibitem{2} {\sc T. Akahori, S. Ibrahim, H. Kikuchi and H. Nawa}, {\em Global dynamics above the ground state energy for the combined power type nonlinear Schrodinger equations with energy critical growth at low frequencies}, Mem. Amer. Math. Soc. 272 (2021), no. 1331, v+130 pp.

%\bibitem{MR2816471}
%{\sc A.~Ambrosetti and D.~Arcoya}, {\em An introduction to nonlinear functional
%  analysis and elliptic problems}, vol.~82 of Progress in Nonlinear
%  Differential Equations and their Applications, Birkh\"{a}user Boston, Ltd.,
% Boston, MA, 2011.

%\bibitem{Berestycki1983}
%{\sc H.~Berestycki and P.-L. Lions}, {\em Nonlinear scalar field equations.
%  {I}. {E}xistence of a ground state}, Arch. Rational Mech. Anal., 82 (1983),
%  pp.~313--345.

\bibitem{4}
{\sc C. G. Boehmer and T. Harko}, {\em Can dark matter be a Bose-Einstein condensate}, Journal of Cosmology and Astroparticle Physics, 2007(06) (2007), 025.

\bibitem{8}
{\sc P. H. Chavanis}, {\em  Mass-radius relation of Newtonian self-gravitating Bose-Einstein condensates with short-range interactions. I. Analytical results}, Physical Review D, 84(4) (2011), 043531.

\bibitem{FV}
{\sc M.~M. Fall and E.~Valdinoci}, {\em Uniqueness and nondegeneracy of
  positive solutions of {$ (-\Delta)^s u+u=u^p$} in {$\Bbb R^N$} when {$s$} is
  close to 1}, Comm. Math. Phys., 329 (2014), pp.~383--404.

%\bibitem{FLS}
%{\sc R.~L. Frank, E.~Lenzmann and L.~Silvestre}, {\em Uniqueness of radial
%solutions for the fractional {L}aplacian}, Comm. Pure Appl. Math., 69 (2016), pp.~1671--1726.

\bibitem{MR619749}
{\sc B.~Gidas and J.~Spruck}, {\em A priori bounds for positive solutions of
  nonlinear elliptic equations}, Comm. Partial Differential Equations, 6
  (1981), pp.~883--901.

\bibitem{MR615628}
{\sc B.~Gidas and J.~Spruck}, {\em Global and local behavior of positive
  solutions of nonlinear elliptic equations}, Comm. Pure Appl. Math., 34
  (1981), pp.~525--598.

\bibitem{GH}
{\sc H.~Groemer}, {\em Geometric applications of Fourier series and spherical
  harmonics}, vol.~61, Cambridge University Press, 1996.

%\bibitem{MR727256}
%{\sc H.~Hofer}, {\em A note on the topological degree at a critical point of
%  mountainpass-type}, Proc. Amer. Math. Soc., 90 (1984), pp.~309--315.


\bibitem{MR969899}
{\sc M.~K. Kwong}, {\em Uniqueness of positive solutions of {$\Delta
  u-u+u^p=0$} in {${\bf R}^n$}}, Arch. Rational Mech. Anal., 105 (1989),
  pp.~243--266.

\bibitem{Lenzmann2009}
{\sc E.~Lenzmann}, {\em Uniqueness of ground states for pseudorelativistic
  {H}artree equations}, Anal. PDE, 2 (2009), pp.~1--27.

\bibitem{Lieb1976/77}
{\sc E.~H. Lieb}, {\em Existence and uniqueness of the minimizing solution of
  {C}hoquard's nonlinear equation}, Studies in Appl. Math., 57 (1976/77),
  pp.~93--105.

\bibitem{LiebLoss.2001}
{\sc E.~H. Lieb and M.~Loss}, {\em Analysis, volume 14 of graduate studies in
  mathematics}, American Mathematical Society, Providence, RI,, 4 (2001).

\bibitem{LMZ}
{\sc X. Li, S. Ma and G. Zhang}, {\em   Existence and qualitative properties of solutions for Choquard equations with a local term}, Nonlinear Analysis: Real World Applications, 45 (2019), pp.~1--25.

\bibitem{lions1987solutions}
{\sc P.-L. Lions}, {\em Solutions of hartree-fock equations for coulomb
  systems}, Communications in Mathematical Physics, 109 (1987), pp.~33--97.

%\bibitem{LiuMoroz} {\sc  Z. Liu and V. Moroz}, {\em Limit profiles for singularly perturbed Choquard equations with local repulsion}, Calc. Var. Partial Differential Equations 61 (2022), Paper No. 160, 59 pp.

\bibitem{louis}
{\sc J.~Louis, J.~Zhang and X.~Zhong}, {\em A global branch approach to
  normalized solutions for the schr\"{o}dinger equation}, arXiv preprint
  arXiv:2112.05869,  (2021).

%\bibitem{Luo}
%{\sc H.~Luo}, {\em Uniqueness and nondegeneracy of ground states for $(-\Delta)^ su+ u= 2 (I_2\star u^ 2) u $ in $\mathbb %{R}^ N $ when $ s $ is close to 1}, arXiv preprint arXiv:2303.14610, 2023.

%\bibitem{luo2022bifurcation}
%{\sc H.~Luo, B.~Ruf and C.~Tarsi}, {\em Bifurcation into spectral gaps for
%  strongly indefinite choquard equations}, arXiv preprint arXiv:2205.02542,
%  (2022).

\bibitem{ma2023asymptotic}
{\sc S. Ma and V. Moroz}, {\em Asymptotic profiles for Choquard equations with combined attractive nonlinearities},
arXiv preprint arXiv:2302.13727, (2023).

\bibitem{Ma2010}
{\sc L.~Ma and L.~Zhao}, {\em Classification of positive solitary solutions of
  the nonlinear {C}hoquard equation}, Arch. Ration. Mech. Anal., 195 (2010),
  pp.~455--467.

\bibitem{MR3568051}
{\sc C.~Mercuri, V.~Moroz and J.~Van~Schaftingen}, {\em Groundstates and
  radial solutions to nonlinear {S}chr\"{o}dinger-{P}oisson-{S}later equations
  at the critical frequency}, Calc. Var. Partial Differential Equations, 55
  (2016), pp.~Art. 146, 58.

%\bibitem{35} {\sc V. Moroz and C. B. Muratov}, {\em Asymptotic properties of ground states of scalar field equations with a vanishing parameter}, J. Eur. Math. Soc. 16 (2014), pp.~1081-1109.

\bibitem{MR3056699}
{\sc V.~Moroz and J.~Van~Schaftingen}, {\em Groundstates of nonlinear
  {C}hoquard equations: existence, qualitative properties and decay
  asymptotics}, J. Funct. Anal., 265 (2013), pp.~153--184.

\bibitem{MR3356947}
\leavevmode\vrule height 2pt depth -1.6pt width 23pt, {\em Existence of
  groundstates for a class of nonlinear {C}hoquard equations}, Trans. Amer.
  Math. Soc., 367 (2015), pp.~6557--6579.

\bibitem{MR3625092}
\leavevmode\vrule height 2pt depth -1.6pt width 23pt, {\em A guide to the
  {C}hoquard equation}, J. Fixed Point Theory Appl., 19 (2017), pp.~773--813.

\bibitem{Pekar}
{\sc S.~Pekar}, {\em Untersuchung \"{u}ber die Elekronentheorie der Kristalle},
  Akedemie Verlag, Berlin, 1954.

\bibitem{QuittnerSouplet.2007}
{\sc P.~Quittner and P.~Souplet}, {\em Superlinear parabolic problems: Blow-up,
  global existence and steady states}, Springer, 2007.

\bibitem{quittner2012symmetry}
\leavevmode\vrule height 2pt depth -1.6pt width 23pt, {\em Symmetry of
  components for semilinear elliptic systems}, SIAM Journal on Mathematical
  Analysis, 44 (2012), pp.~2545--2559.

%\bibitem{MR0301587}
%{\sc P.~H. Rabinowitz}, {\em Some global results for nonlinear eigenvalue
%  problems}, J. Functional Analysis, 7 (1971), pp.~487--513.

\bibitem{39} {\sc R. Ruffini and S. Bonazzola}, {\em Systems of self-gravitating particles in general relativity and the concept of an equation of state}, Physical Review 187 (1969), pp.~1767-1783.

\bibitem{MR2679375}
{\sc D.~Ruiz}, {\em On the {S}chr\"{o}dinger-{P}oisson-{S}later system:
  behavior of minimizers, radial and nonradial cases}, Arch. Ration. Mech.
  Anal., 198 (2010), pp.~349--368.

\bibitem{MR1677740}
{\sc P.~Tod and I.~M. Moroz}, {\em An analytical approach to the
  {S}chr\"{o}dinger-{N}ewton equations}, Nonlinearity, 12 (1999), pp.~201--216.

\bibitem{Vaira2013}
{\sc G.~Vaira}, {\em Existence of bound states for {S}chr\"{o}dinger-{N}ewton
  type systems}, Adv. Nonlinear Stud., 13 (2013), pp.~495--516.

\bibitem{MR2492602}
{\sc J.~Wei and M.~Winter}, {\em Strongly interacting bumps for the
  {S}chr\"{o}dinger-{N}ewton equations}, J. Math. Phys., 50 (2009), pp.~012905,
  22.

\bibitem{45}
{\sc X. Z. Wang}, {\em   Cold Bose stars: self-gravitating Bose-Einstein condensates}, Physical Review D, 64(12) (2001), 124009.

%\bibitem{Xiang}
%{\sc C.~Xiang}, {\em  Quantitative properties on the steady states to a Schr\"{o}dinger-Poisson-Slater System}, Acta. Math. Sin.-English Ser. 31 (2015), pp.~1845-1856.

\end{thebibliography}
\end{document}